\documentclass[12pt]{amsart}

\usepackage[pdftex,pagebackref,colorlinks,citecolor=blue,urlcolor=blue,linkcolor=magenta]{hyperref}
\usepackage{amssymb}
\usepackage{graphicx,color}
\usepackage{amsmath}
\usepackage{comment}
\usepackage{MnSymbol}
\usepackage{mathtools}
\usepackage{cleveref}

\usepackage{multicol}
\usepackage{caption}
\usepackage[font=small]{subcaption}
\usepackage[normalem]{ulem}
\newcommand{\middlewave}[1]{\raisebox{0.5em}{\uwave{\hspace{#1}}}}
\usepackage{tikz}
\usetikzlibrary{snakes}

\newtheorem{thmx}{Theorem}

\newtheorem{theorem}{Theorem}[section]
\newtheorem{proposition}[theorem]{Proposition}
\newtheorem{lemma}[theorem]{Lemma}
\newtheorem{corollary}[theorem]{Corollary}
\theoremstyle{definition}
\newtheorem{definition}[theorem]{Definition}

\newtheorem{example}[theorem]{Example}
\newtheorem{problem}[theorem]{Problem}

\DeclareMathOperator{\RR}{\mathbb{R}}
\DeclareMathOperator{\ZZ}{\mathbb{Z}}

\title{Triangulations of Cosmological Polytopes}

\date{\today}

\author{Martina Juhnke-Kubitzke}
\address{Department of Mathematics, University of Osnabr\"{u}ck, Albrechtstra\ss e 28a, 49076 Osnabr\"uck, Germany}
\email{juhnke-kubitzke@uni-osnabrueck.de}

\author{Liam Solus}
\address{Department of Mathematics, KTH Royal Institute of Technology, SE-100 44 Stockholm, Sweden}
\email{solus@kth.se}

\author{Lorenzo Venturello}
\address{Department of Mathematics, Universit\`{a} di Pisa, L.go B. Pontecorvo, 5, 56127 Pisa, Italy }
\email{lorenzo.venturello@unipi.it}

\begin{document}

\begin{abstract}
A cosmological polytope is defined for a given Feynman diagram, and its canonical form may be used to compute the contribution of the Feynman diagram to the wavefunction of certain cosmological models. 
Given a subdivision of a polytope, its canonical form is obtained as a sum of the canonical forms of the facets of the subdivision. 
In this paper, we identify such formulas for the canonical form via algebraic techniques. 
It is shown that the toric ideal of every cosmological polytope admits a Gr\"obner basis with a squarefree initial ideal, yielding a regular unimodular triangulation of the polytope. 
In specific instances, including trees and cycles, we recover graphical characterizations of the facets of such triangulations that may be used to compute the desired canonical form. 
For paths and cycles, these characterizations admit simple enumeration. 
Hence, we obtain formulas for the normalized volume of these polytopes, extending previous observations of K\"uhne and Monin. 
\end{abstract}

\maketitle

\section{Introduction}
\label{sec: intro}

Arkani-Hamed, Benincasa and Postnikov \cite{arkani2017cosmological} introduced the cosmological polytope $\mathcal{C}_G$ of an undirected, connected graph $G = (V,E)$, where $V$ is the finite set of \emph{vertices (or nodes)} of $G$ and $E$ is its finite collection of \emph{edges}; i.e., pairs $ij$ for some $i,j\in V$.
When we would like to emphasize that $V$ and $E$ are, respectively, the vertex and edge set of $G$, we may write $V(G)$ and $E(G)$, respectively. 
We will use $ij$ to denote an undirected edge between $i$ and $j$, and $(i,j)$ to denote a directed edge $i\rightarrow j$ when edge directions are needed.

We work in the finite real-Euclidean space $\RR^{|V|+|E|}$ with standard basis vectors $x_i$ and $x_e$ for all $i\in V$, $e\in E$. 
The \emph{cosmological polytope} $\mathcal{C}_G$ of $G$ is 
\[
        \mathcal{C}_G=\text{conv}\{x_i+x_j-x_e,x_i-x_j+x_e,-x_i+x_j+x_e ~:~ e = ij\in E\}.
\]
It is only required that the graph $G$ is connected and undirected with a finite set of vertices and edges. 
For instance, $G$ need not be simple. 
In \cite{kuhne2022faces}, the authors work with a slight generalization of the definition of $\mathcal{C}_G$ that allows for $G$ to be disconnected. 
For the purposes of this paper, however, we will consider only connected $G$.

In the physical context, the graph $G$ can be interpreted as a Feynman diagram, in which case the cosmological polytope provides a geometric model for the computation of the contribution of the Feynman diagram represented by $G$ to the so-called \emph{wavefunction of the universe} \cite{arkani2017cosmological}. 
Recent works study the physics of scattering amplitudes via a generalization of convex polytopes called \emph{positive geometries} \cite{arkani2017positive}. 
This connection arises via a unique differential form of the positive geometry that has only logarithmic singularities along its boundary. 
This form is termed its \emph{canonical form} \cite{arkani2017positive}.
In the case of a cosmological polytope $\mathcal{C}_G$, the canonical form provides a formula for computing the contribution of the Feynman diagram $G$ to certain wavefunctions of interest. 

One way to compute the canonical form $\Omega_P$ of a polytope $P$ is as a sum of the canonical forms of the facets $S_1,\ldots, S_m$ of a subdivision of $P$ \cite{lam2022invitation}; i.e.,
\begin{equation}
    \label{eqn: canonical form}
    \Omega_P = \Omega_{S_1} + \cdots + \Omega_{S_m}. 
\end{equation}
This technique has been applied successfully in several situations \cite{benincasa2022physical,galashin2020parity,herrmann2021positive,kojima2020sign,mohammadi2021triangulations}.  

In the case of cosmological polytopes, it was observed in \cite{arkani2017cosmological} for specific examples of $G$ that special subdivisions of the cosmological polytope correspond to classical physical theories for the computation of the contribution of a Feynman diagram to a wavefunction. 
This observation suggests that subdivisions of the cosmological polytope may correspond to physical theories for the computation of wavefunctions, and hence motivates a search for nice subdivisions of cosmological polytopes that hold for any graph $G$.  
Such subdivisions have the potential to provide new physical theories for wavefunction computations. 
The investigation of subdivisions that hold for any graph $G$ was left as future work in \cite{arkani2017cosmological}.
In this paper, we provide such subdivisions by way of algebraic techniques.

The cosmological polytope is a \emph{lattice polytope}; i.e., the convex hull of a finite collection of points in $\ZZ^{|V|+|E|}$.
From this perspective, it is perhaps most natural to investigate its subdivisions into \emph{unimodular} simplices, i.e., lattice simplices of minimum Euclidean volume. 
Such subdivisions have the advantage that each summand in \eqref{eqn: canonical form} has the relatively simple form of a rational function
\[
\frac{\omega}{f_1\cdots f_r},
\]
where $f_1,\ldots, f_r$ are the facet-defining equations of the simplex and $\omega$ is a regular form on the associated positive geometry \cite{arkani2015positive}. 
Hence, it is desirable to observe not only the existence of unimodular triangulations of $\mathcal{C}_G$ but also to provide a complete description of their facets. 

In this paper, we initiate the study of this problem via algebraic techniques. 
We compute a family of Gr\"obner bases of the toric ideal for the cosmological polytope $\mathcal{C}_G$ for any connected, undirected graph $G$, each of which has a squarefree initial ideal. 
It is known that the initial terms of such a Gr\"obner basis correspond to the minimal non-faces of a regular unimodular triangulation of $\mathcal{C}_G$ \cite{sturmfels1996grobner}. 
Hence, we obtain the following main result:
\begin{thmx} (\Cref{cor: triangulation})
    \label{thm: main}
    The cosmological polytope $\mathcal{C}_G$ of any undirected, connected graph $G$ has a regular unimodular triangulation. 
\end{thmx}

An analogous result holds for another family of lattice polytopes associated to graphs; namely \emph{symmetric edge polytopes} \cite{HJM,FirstSEP}. Interestingly, these are strongly related to cosmological polytopes, as the symmetric edge polytope of a graph appears as a specific projection of a facet of the cosmological polytope of the same graph.

The identified Gr\"obner bases provide the minimal non-faces of triangulations, from which achieving a facet description is a non-trivial task. One of the contributions of this article is to obtain such a characterization for the cosmological polytope of notable families of graphs.
\begin{thmx}
    For specific choices of a term order we obtain a facet description of a regular unimodular triangulation of the cosmological polytope $\mathcal{C}_G$, when $G$ is:
\begin{itemize}
    \item[-] a path  (\Cref{lem: path tri facets});
    \item[-] a cycle (\Cref{thm: cycle facets});
    \item[-] a tree (\Cref{thm: tree facet characterization}).
\end{itemize}
\end{thmx}
In the case of paths and cycles these characterizations are of a relatively simple form that allows for enumeration. 
We thereby obtain formulas for the normalized volume of $\mathcal{C}_G$ in these two  cases. While, for paths, we recover the formula identified in \cite{kuhne2022faces}, for the cycle, the normalized volume of $\mathcal{C}_G$ was previously unknown. Indeed, our methods enable us to show the following simple formula:
\begin{thmx} (\Cref{thm : vol cycle})
    The cosmological polytope $\mathcal{C}_{C_n}$ of the $n$-cycle $C_n$ has normalized volume
    \[
        \textnormal{Vol}(\mathcal{C}_{C_n})=4^n-2^n.
    \]
\end{thmx}
While the normalized volume of these polytopes provides us with information on the number of summands in the formula \eqref{eqn: canonical form} for computing $\Omega_{\mathcal{C}_G}$, the explicit description of the facets that we obtain for trees and cycles given in Theorems~\ref{lem: path tri facets}, \ref{thm: cycle facets}, and~\ref{thm: tree facet characterization} allows for the exact computation of this canonical form.  
Theorem~\ref{thm: main} suggests that such characterizations should be feasible for more general graphs via further analysis of the Gr\"obner bases identified in this paper.

\section{Gr\"obner Bases for the toric ideal of $\mathcal{C}_G$}
\label{sec: grobner basis}
In this section, we describe a family of Gr\"{o}bner bases for the toric ideal of a cosmological polytope with the property that the corresponding initial ideals are squarefree. 
%
We start with some definitions. 
For any undirected graph $G$ with vertex set $V$ and edge set $E$, we define a polynomial ring in $|V|+4|E|$ variables, each corresponding to a lattice point of $\mathcal{C}_G$. More precisely, we introduce three families of variables:
\begin{itemize}
    \item[-] A variable $z_k$, for every $k\in V\cup E$. We refer to these as \emph{$z$-variables}.
    \item[-] Variables $y_{ije}$ and $y_{jie}$, for every edge $e=ij\in E$. We refer to these as \emph{$y$-variables}. 
    \item[-] A variable $t_e$ for every edge $e\in E$. We refer to these as \emph{$t$-variables}. 
\end{itemize}
Let $R_G$ be the polynomial ring in these $|V|+4|E|$ many variables, with coefficients in a field $K$, and consider the surjective homomorphism of $K$-algebras defined by
\begin{align*}
    \varphi_G: R_G&\to K[\mathbf{w}^{p} ~ : ~ p\in \mathcal{C}_G\cap \mathbb{Z}^{V\cup E}]\\
    z_k &\mapsto w_k\\
    y_{ije} &\mapsto w_iw_j^{-1}x_{e}\\
    y_{jie} &\mapsto w_i^{-1}w_j w_{e} \\ 
    t_{e} &\mapsto w_iw_jw_{e}^{-1}.
\end{align*}
The ideal $I_{\mathcal{C}_G}\coloneqq \ker(\varphi_G)$ is the \emph{toric ideal} of $\mathcal{C}_G$. Observe that variables in $R_G$ correspond to lattice points of $\mathcal{C}_G$.
When the graph $G$ is understood from the context, we may simply write $\varphi$ for $\varphi_G$.
We now define some distinguished binomials in $I_{\mathcal{C}_G}$, which will be elements of a Gr\"{o}bner basis for this ideal.


\begin{definition}
    We define two types of pairs of directed subgraphs of $G$.
    \begin{itemize}
        \item[(i)] Let $P$ be a path in $G$, with $E(P)=\{i_1i_2,i_2i_3,\dots,i_{k-1}i_k\}$. For any partition $(P_1,P_2)$ of $E(P)$ into two nonempty blocks we consider $E_1 = \{i_{j}\to i_{j+1} ~ : ~ i_{j}i_{j+1}\in E(P_1)\}$, and $E_2 = \{i_{j+1}\to i_{j} ~ : ~ i_{j}i_{j+1}\in E(P_2)\}$. The pair $(E_1,E_2)$ is called a \emph{zig-zag pair} of $G$. 
    
        Moreover, we define the  \emph{terminal vertices} of $(E_1,E_2)$ to be $v_1=i_k$ and $v_2=i_1$.
        \item[(ii)] Let $C$ be a cycle in $G$, with $E(C)=\{i_1i_2,i_2i_3,\dots,i_{k-1}i_k, i_ki_1\}$. For any partition $(C_1,C_2)$ of $E(C)$ into two blocks (with one possibly empty) we consider $E_1 = \{i_{j}\to i_{j+1} ~ : ~ i_{j}i_{j+1}\in E(C_1)\}$, and $E_2 = \{i_{j+1}\to i_{j} ~ : ~ i_{j}i_{j+1}\in E(C_2)\}$. The pair $(E_1,E_2)$ is called a \emph{cyclic pair} of $G$.
    \end{itemize}
\end{definition}

\begin{figure}[t]
    \centering
    \includegraphics[scale=0.8]{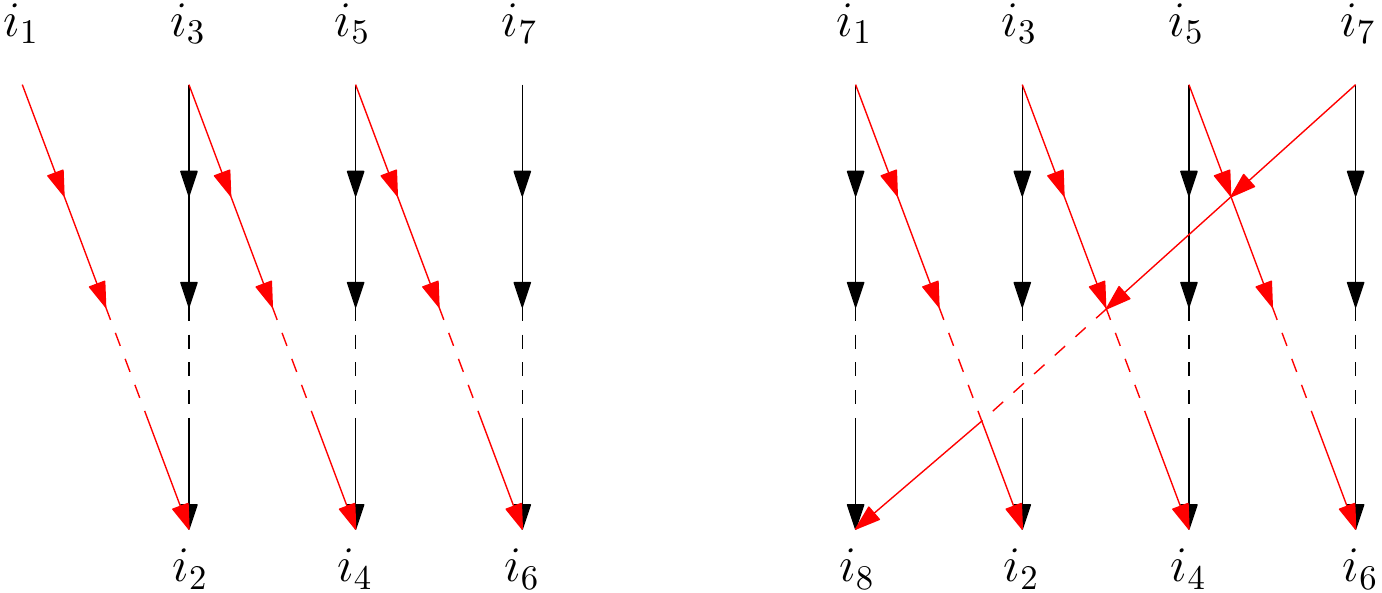}
    \caption{A zig-zag pair and a cyclic pair $(E_1,E_2)$. The edges in $E_1$ are drawn in red.}
    \label{fig:pairs}
\end{figure}

\begin{definition}
    For every zig-zag pair $(E_1,E_2)$ we define the \emph{zig-zag binomial}
   \[
        b_{E_1,E_2} =
       \displaystyle z_{v_1}\prod_{e=i\to j\in E_1}y_{ije}\prod_{e=i \to j\in E_2} z_e - z_{v_2}\prod_{e=i \to j\in E_2}y_{ije}\prod_{e=i \to j\in E_1} z_e.
    \]   
    For every cyclic pair $(E_1,E_2)$ we define the \emph{cyclic binomial}
    \[            
    b_{E_1,E_2} = \displaystyle\prod_{e=i\to j\in E_1}y_{ije}\prod_{e=i\to j\in E_2} z_e - \prod_{e=i \to j\in E_2}y_{ije}\prod_{e= i\to j\in E_1} z_e.
    \]
    In the case that either $E_1 = \emptyset$ or $E_2 = \emptyset$, we call the resulting cyclic binomial a \emph{cycle binomial}. In particular, cycle binomials consist of one monomial containing only $y$-variables and one containing only $z$-variables.
\end{definition}

\begin{definition}\label{def: basis}
	We define the following collection of binomials in $R_G$.
	\begin{align*}
		B_G=\{& \underline{y_{ije}y_{jie}}-z_e^2,~ \underline{y_{ije}t_{e}}-z_i^2,~  \underline{y_{jie}t_{e}}-z_j^2,\\
		& \underline{y_{ije}z_j} - z_iz_e,~ \underline{y_{jie}z_i} - z_jz_e,~ \underline{t_{e}z_e} - z_iz_j ~:~ e=ij\in E(G)\} \cup\\
		& \{b_{E_1,E_2} ~ : ~ (E_1,E_2) \text{ is a zig-zag pair or a cyclic pair of $G$}\}.
	\end{align*}
	We call the $6|E|$ many binomials contained in the first set the \emph{fundamental binomials}.
\end{definition}
\begin{definition}\label{def : good}
	A term order on $R_G$ is called a \emph{good} term order if the leading terms of the degree two binomials in $B_G$ are the elements underlined in \Cref{def: basis}, and the leading terms of cycle binomials are the monomials containing only $y$-variables.
\end{definition}
We observe that good term orders exist. 
For instance we can consider any lexicographic term order for which $y$-variables and $t$-variables are larger than any $z$-variable.
We will now show that, for any undirected, connected graph $G$, the set $B_G$ is a Gr\"obner basis for $I_{\mathcal{C}_G}$ with respect to any good term order. 
To do so, we require a few lemmas. 

\begin{lemma}\label{lem: yij(e)}
    Let $b = \mathbf{m}_1 - \mathbf{m}_2$ be a binomial in $I_{\mathcal{C}_G}$, and assume that no variable divides both $\mathbf{m}_1$ and $\mathbf{m}_2$. If $t_{e}|\mathbf{m}_1$, for some edge $e=ij$ of $G$, then $\mathbf{m}_1$ is divisible by the leading term of a fundamental binomial in $B_G$ with respect to any good term order.
\end{lemma}
\begin{proof}
     Assume $t_{e}|\mathbf{m}_1$ and recall that $\varphi(t_{e})=w_iw_jw_e^{-1}$. Since $b\in I_{\mathcal{C}_G}$, we have that $\varphi(\mathbf{m}_1)=\varphi(\mathbf{m}_2)$. In particular, the variable $w_{e}$ either appears in the Laurent monomial $\varphi(\mathbf{m}_2)$ with a negative exponent, or it appears in $\varphi(\mathbf{m}_1/t_{e})$ with positive exponent. The first case contradicts the fact that $b$ is irreducible: indeed, since $t_{e}$ is the only variable for which the variable $w_e$ appears in $\varphi(t_{e})$ with a negative exponent, this would imply that $t_{e}|\mathbf{m}_2$.\\
     In the second case we have that one of the variables $v$ for which $w_e$ appears in $\varphi(v)$ with a positive exponent must divide $\mathbf{m}_1$. These are either $z_e$, $y_{jie}$ or $y_{ije}$. 
     This concludes the proof, since the monomials $t_{e}z_e$, $y_{jie}t_{e}$ and $y_{ije}t_{e}$ are leading terms of some binomial in $B_G$ with respect to the chosen term order. 
\end{proof}

We now associate to any binomial $\mathbf{m}_1 - \mathbf{m}_2$ a pair of directed subgraphs $(\overrightarrow{G_1},\overrightarrow{G_2})$ of $G$ in the following way:
For any variable $y_{ije}$ which divides $\mathbf{m}_1$ (respectively $\mathbf{m}_2$) the graph $\overrightarrow{G_1}$ (respectively $\overrightarrow{G_2}$) contains the vertices $i$ and $j$ and a number of directed edges from $i$ to $j$ equal to the degree of $y_{ije}$ in $\mathbf{m}_1$ (respectively $\mathbf{m}_2$).
\begin{definition}
	For a directed graph $\overrightarrow{G}$ and a vertex $i\in V(\overrightarrow{G})$ we define $\deg_{\overrightarrow{G}}(i)=\text{outdeg}_{\overrightarrow{G}}(i)-\text{indeg}_{\overrightarrow{G}}(i)$, where $\text{outdeg}_{\overrightarrow{G}}(i) = |\{j\in V(\overrightarrow{G}) : (i,j)\in E(\overrightarrow{G})\}|$ and $\text{indeg}_{\overrightarrow{G}}(i) = |\{j\in V(\overrightarrow{G}) : (j,i)\in E(\overrightarrow{G})\}|$. If $\deg_{\overrightarrow{G}}(i)>0$, we call $i$ is  \emph{positive} vertex of  $\overrightarrow{G}$.  If $\deg_{\overrightarrow{G}}(i)<0$, we call $i$ a \emph{negative} vertex of  $\overrightarrow{G}$.
\end{definition}

\begin{lemma}\label{lem: technical}
    Let $b = \mathbf{m}_1 - \mathbf{m}_2$ be an irreducible binomial in $I_{\mathcal{C}_G}$, and let $(\overrightarrow{G_1},\overrightarrow{G_2})$ be the associated pair of directed graphs. Assume that no leading term of a fundamental binomial in $B_G$ with respect to a good term order divides $\mathbf{m}_1$ or $\mathbf{m}_2$.
    Then:
    \begin{itemize}
        \item[(1)] If $\deg_{\overrightarrow{G_1}}(i)<0$, then $i\in V(\overrightarrow{G_1})\cap V(\overrightarrow{G_2})$.  Moreover, if $i\in V(\overrightarrow{G_1})\cap V(\overrightarrow{G_2})$, then $\deg_{\overrightarrow{G_1}}(i)= \deg_{\overrightarrow{G_2}}(i)$. 
        \item[(2)] If $i\in V(\overrightarrow{G_1})\setminus V(\overrightarrow{G_2})$ and $\deg_{\overrightarrow{G_1}}(i)>0$ ($i\in V(\overrightarrow{G_2})\setminus V(\overrightarrow{G_1})$ and $\deg_{\overrightarrow{G_2}}(i)>0 $, respectively), then $z_i| \mathbf{m}_2$ ($z_i| \mathbf{m}_1$, respectively).
       \item[(3)]  If $e\in E(\overrightarrow{G_1})\setminus E(\overrightarrow{G_2})$ ($e\in E(\overrightarrow{G_2})\setminus E(\overrightarrow{G_1})$ respectively), then $z_e|\mathbf{m}_2$ ($z_e|\mathbf{m}_1$ respectively).    \end{itemize}
\end{lemma}
\begin{proof}
     (1) If $\deg_{\overrightarrow{G_1}}(i)<0$, then the degree of $w_i$ in $\varphi(\mathbf{m}_1)$ is negative. Since $b\in I_{\mathcal{C}_G}$, the degree of $w_i$ in $\varphi(\mathbf{m}_2)$ is also negative. As the only variables $v$ such that $w_i$ has negative exponent in $\varphi(v)$ are of the form $y_{jie}$ for some vertex $j$ and edge $e$, the claim follows. Since $i\in V(\overrightarrow{G_1})$ there is at least one edge incident to $i$ in $\overrightarrow{G_1}$. We have then that $\mathbf{m}_1$ is divisible by a variable $y_{ije}$ or $y_{jie}$, for some vertex $j$ and edge $e$. In particular, we conclude that $z_i$ does not divide $\mathbf{m}_1$ as we assumed that $y_{ije}z_i$ does not divide $\mathbf{m}_1$. By symmetry $z_i$ does not divide $\mathbf{m}_2$. It follows that $\deg_{\overrightarrow{G_1}}(i)$ equals the degree of the variable $w_i$ in $\varphi(\mathbf{m}_1)$ and that $\deg_{\overrightarrow{G_2}}(i)$ equals the degree of the variable $w_i$ in $\varphi(\mathbf{m}_2)$. Since $b\in I_{\mathcal{C}_G}$, we conclude that $\deg_{\overrightarrow{G_1}}(i)=\deg_{\overrightarrow{G_2}}(i)$. 
     
     
     (2) As in the previous case, the number $\deg_{\overrightarrow{G_1}}(i)$ equals the degrees of the variable $w_i$ in $\varphi(\mathbf{m}_1)$ and $\varphi(\mathbf{m}_2)$. Since $i\notin V(\overrightarrow{G_2})$, the only variable in $\mathbf{m}_2$ which contributes to a positive degree in $\varphi(\mathbf{m}_2)$ is $z_i$. 
     
     (3) Again, since $b\in I_{\mathcal{C}_G}$, the degrees of $w_e$ in $\varphi(\mathbf{m}_1)$ and $\varphi(\mathbf{m}_2)$ coincide. By \Cref{lem: yij(e)} the variable $t_e$ does not divide neither $\mathbf{m}_1$ nor $\mathbf{m}_2$, and this is the only variable $v$ such that $w_e$ has a negative degree in $\varphi(v)$. Hence $w_e$ has positive degree in both $\varphi(\mathbf{m}_1)$ and $\varphi(\mathbf{m}_2)$. Since $e\notin E(\overrightarrow{G_2})$, the only variable which contributes to a positive degree of $w_e$ in $\varphi(\mathbf{m}_2)$ is $z_e$.
\end{proof}

The following lemma collects some simple properties of directed acyclic graphs that will be of use.

\begin{lemma}\label{lem: directed acyclic properties}
    Let $H$ be a directed acyclic graph, with at least one edge and no isolated vertices. Then $H$ has at least a positive and a negative vertex. Moreover, for every positive vertex $i\in V(H)$ there exists a negative vertex $j\in V(H)$ such that $H$ contains a directed path from $i$ to $j$, and for every negative vertex $j\in V(H)$ there exists a positive vertex $i\in V(H)$ such that $H$ contains a directed path from $i$ to $j$.
\end{lemma}

\begin{proof}
Every directed acyclic graph with at least one edge has at least one sink and at least one source node.  Since sinks are positive vertices and sources are negative vertices the first claim holds.  The second claim follows from the fact that every vertex in the directed acyclic graph has at least one descendant that is a sink and every vertex has at least one source node as an ancestor.
\end{proof}
We are now ready to prove the main result of the section.
\begin{theorem}
\label{thm: gb}
    The set $B_G$ is a Gr\"{o}bner basis of $I_{\mathcal{C}_G}$ with respect to every good term order.
\end{theorem}
\begin{proof}
     Let $b = \mathbf{m}_1 - \mathbf{m}_2$ be a binomial in $I_{\mathcal{C}_G}$, and assume that no variable divides both $\mathbf{m}_1$ and $\mathbf{m}_2$. We prove that there exists a binomial $f\in B_G$ such that $\text{lt}(f)| \mathbf{m}_1$ or $\text{lt}(f)| \mathbf{m}_2$. This shows that any binomial in $I_{\mathcal{C}_G}$ can be reduced by an element of $B_G$. Since this reduction step produces another binomial and the sequence of reduction terminates, it must terminate with the zero polynomial. In particular, all $S$-polynomials obtained from  a generating set of binomials of $I_{\mathcal{C}_G}$ reduce to zero, which implies that $B_G$ is a Gr\"{o}bner basis.
     
     If the leading term of a fundamental binomial in $B_G$ divides either $\mathbf{m}_1$ or $\mathbf{m}_2$, then we conclude.
     
     Assume that no leading term of a fundamental binomial in $B_G$ divides either $\mathbf{m}_1$ or $\mathbf{m}_2$. In particular, by \Cref{lem: yij(e)}, no variable of the form $t_{e}$ divides either $\mathbf{m}_1$ or $\mathbf{m}_2$. Consider the pair $(\overrightarrow{G_1},\overrightarrow{G_2})$ of directed subgraphs of $G$ associated with $\mathbf{m}_1$ and $\mathbf{m}_2$.
     
If $\overrightarrow{G_1}$ ($\overrightarrow{G_2}$ respectively) has a directed cycle $C$, then by construction $\mathbf{m}_1$ ($\mathbf{m}_2$ respectively) is divisible by the monomial $\prod_{\overset{\to}{e}=(i,j)\in E(C)}y_{ije}$ which is the leading term of a cycle binomial by definition of good term order and so we conclude.
     
     Assume that both $\overrightarrow{G_1}$ and $\overrightarrow{G_2}$ are directed acyclic. Since $b$ is irreducible, $\overrightarrow{G_1}$ and $\overrightarrow{G_2}$ do not have any common directed edge, as those would correspond to variables which divide both $\mathbf{m}_1$ and $\mathbf{m}_2$.
     
     Suppose there is a positive vertex $i$ in $\overrightarrow{G_1}$ such that $i\in V(\overrightarrow{G_1})\setminus V(\overrightarrow{G_2})$. Observe that by \Cref{lem: technical} (2) this implies that $z_i|\mathbf{m}_2$. 
     We let $i_1=i$ and $j_1$ be a negative vertex of $\overrightarrow{G_1}$ such that there is a directed path from $i_1$ to $j_1$. 
     By \Cref{lem: technical} (1), $j_1$ is a negative vertex of $\overrightarrow{G_2}$ as well. 
     By \Cref{lem: directed acyclic properties}, there exists a positive vertex $i_2$ of $\overrightarrow{G_2}$ such that there is a directed path in $\overrightarrow{G_2}$ from $i_2$ to $j_2$. If $i_2\in V(\overrightarrow{G_1})$, by \Cref{lem: technical} (1), we have $\deg_{\overrightarrow{G_1}}(i_2)=\deg_{\overrightarrow{G_2}}(i_2)>0$. 
     We can then iterate this procedure until one of the following possibilities occurs:
     
     {\sf Case 1:} $i_k\notin V(\overrightarrow{G_1})$. In this case, let $E_1$ be the union of the directed edges of the directed paths from $i_t$ to $j_t$, for $t=1,\dots,k-1$ and $E_2$ be the union of the directed edges of the directed paths from $i_{t+1}$ to $j_t$, for $t=1,\dots,k-1$. Hence $(E_1,E_2)$ is a zig-zag pair. By definition of the graphs $(\overrightarrow{G_1},\overrightarrow{G_2})$ we have that $\prod_{\overset{\to}{e}=(i,j)\in E_1}y_{ije}$ divides $\mathbf{m}_1$ and $\prod_{\overset{\to}{e}=(i,j)\in E_2}y_{ije}$ divides $\mathbf{m}_2$. Moreover, by \Cref{lem: technical} (2), we have that $z_{i_k}|\mathbf{m}_1$ and that $z_{i_1}|\mathbf{m}_2$. Finally, by \Cref{lem: technical} (3), $\prod_{\overset{\to}{e}=(i,j)\in E_2} z_e$ divides $\mathbf{m}_1$ and $\prod_{\overset{\to}{e}=(i,j)\in E_1} z_e$ divides $\mathbf{m}_2$. 
      In particular, $\mathbf{m}_1$ and $\mathbf{m}_2$ are divisible by the two monomials of $b_{E_1,E_2}$, the binomial corresponding to the zig-zag pair $(E_1,E_2)$.
      
     {\sf Case 2:} $i_k=i_{\ell}$, for some ${\ell}<k$. In this case, let $E_1$ be the union of the directed edges of the directed paths from $i_{t}$ to $j_{t}$, for $t=\ell,\dots,k-1$ and $E_2$ be the union of the directed edges of the directed paths from $i_{t+1}$ to $j_t$, for $t=\ell,\dots,k-1$ together with the directed edges from $i_{k}$ to $j_{\ell}$. The pair $(E_1,E_2)$ is a cyclic pair. Again by definition of the graphs $(\overrightarrow{G_1},\overrightarrow{G_2})$, we have that $\prod_{\overset{\to}{e}=(i,j)\in E_1}y_{ije}$ divides $\mathbf{m}_1$ and $\prod_{\overset{\to}{e}=(i,j)\in E_2}y_{ije}$ divides $\mathbf{m}_2$. Moreover, by \Cref{lem: technical} (3), $\prod_{\overset{\to}{e}=(i,j)\in E_2} z_e$ divides $\mathbf{m}_1$ and $\prod_{\overset{\to}{e}=(i,j)\in E_1} z_e$ divides $\mathbf{m}_2$. In particular, $\mathbf{m}_1$ and $\mathbf{m}_1$ are divisible by the two monomials of $b_{E_1,E_2}$, the binomial corresponding to the cyclic pair $(E_1,E_2)$. This finishes Case 2.
     
     If there is a positive vertex $i$ in $\overrightarrow{G_2}$ such that  $i\in V(\overrightarrow{G_2})\setminus V(\overrightarrow{G_1})$, we can conclude by the same argument as above. 
     
     Suppose now that for all vertices $i$ with $\deg_{\overrightarrow{G_1}}(i)>0$ we have that $i\in V(\overrightarrow{G_2})$ and for all vertices $i$ with $\deg_{\overrightarrow{G_2}}(i)>0$ we have that $i\in V(\overrightarrow{G_1})$. We initialize $i_1$ to be any of the vertices with $\deg_{\overrightarrow{G_1}}(i)>0$ and, as in the previous case, we start constructing disjoint directed paths from $i_t$ to $j_{t}$ in $\overrightarrow{G_1}$ and from $i_{t+1}$ to $j_{t}$ in $\overrightarrow{G_2}$. Since the graphs $\overrightarrow{G_1}$ and $\overrightarrow{G_2}$ are finite there exists $k$ such that $i_k=i_{\ell}$ for some $\ell<k$. Let $E_1$ be the union of the directed edges of the directed paths from $i_{t}$ to $j_{t}$, for $t=\ell,\dots,k-1$ and $E_2$ be the union of the directed edges of the directed paths from $i_{t+1}$ to $j_t$, for $t=\ell,\dots,k-1$ together with the directed edges from $i_{k}$ to $j_{\ell}$. The pair $(E_1,E_2)$ is a cyclic pair. Following verbatim Case 2 we obtain that $\mathbf{m}_1$ and $\mathbf{m}_2$ are divisible by the two monomials of $b_{E_1,E_2}$, the binomial corresponding to the cyclic pair $(E_1,E_2)$.
     This completes the proof.
\end{proof}

\subsection{Regular unimodular triangulations}
\label{subsec: reg tri}
Initial ideals of the toric ideal $I_A$ of an integral point configuration $A$ are in correspondence with regular triangulations of the convex hull of $A$ into lattice simplices (using no additional vertices). More precisely, the radical of any initial ideal of $I_A$ is the \emph{Stanley-Reisner} ideal of a regular triangulation of $\text{conv}(A)$, i.e., the squarefree monomial ideal generated by all monomials corresponding to non-faces of the triangulation (see \cite[Theorem 8.3]{sturmfels1996grobner} or \cite[Section 9.4]{DRS}). Moreover, all regular triangulations of $\text{conv}(A)$ can be obtained in this way. Triangulations corresponding to initial ideals which are squarefree are \emph{unimodular}, meaning that the number of facets equals the normalized volume of $\text{conv}(A)$. Since by \Cref{thm: gb} the set  $B_G$ is a Gr\"{o}bner basis of $I_{\mathcal{C}_G}$  with respect to any good term order with only squarefree initial terms, we obtain the following corollary.

\begin{corollary}
\label{cor: triangulation}
    Let $G$ be any graph. The cosmological polytope $\mathcal{C}_G$ has a regular unimodular triangulation.
\end{corollary}

Corollary~\ref{cor: triangulation} provides the existence of the desired subdivisions of $\mathcal{C}_G$ for any $G$.  
While the result is constructive, the presentation of the resulting triangulations is in the form of their minimal non-faces. 
In order to apply the formula in \eqref{eqn: canonical form} to compute the canonical forms $\Omega_{\mathcal{C}_G}$, we require a description of the triangulations in terms of their facets. 
In the coming sections, we give such characterizations for families of $G$. 
To derive these results we will use some observations that can be seen to hold for all regular unimodular triangulations derived from good term orders for any graph $G$. 
In this subsection, we collect these results and the relevant notation that will be used throughout the remaining sections. 

We start by introducing some notation. 
In the following, let us assume that we have a graph $G=(V,E)$ and a good term order. 
By \Cref{thm: gb} a Gr\"obner basis with squarefree initial ideal is given by the fundamental binomials, the zig-zag binomials and the cyclic binomials.
Since the cosmological polytope $\mathcal{C}_{G}$ has dimension $|V| + |E| -1$, the corresponding regular unimodular triangulation has facets given by all $(|V|+|E|)$-subsets of the variables $y_{ije}, y_{jie}, t_e, z_e, z_i$ that do not contain any leading term of the binomials in this Gr\"obner basis. 

The fundamental binomials imply that certain 2-subsets of variables cannot be contained in the facets.  
These 2-subsets to be avoided for each edge $e=ij\in E$ correspond to the edges of the following graph:
\begin{center}
    \begin{tikzpicture}[thick,scale=0.4]
        \tikzset{decoration={snake,amplitude=.4mm,segment length=2mm, post length=0mm,pre length=0mm}}
	
        \node[circle, draw = black!100, fill=black!00, inner sep=2pt, minimum width=2pt] (z1) at (-2,0) {\scriptsize $z_i$};
        \node[circle, draw = black!100, fill=black!00, inner sep=2pt, minimum width=2pt] (z12) at (4,0) {\scriptsize $t_{e}$};
        \node[circle, draw = black!100, fill=black!00, inner sep=2pt, minimum width=2pt] (z2) at (10,0) {\scriptsize $z_{j}$};
        \node[circle, draw = black!100, fill=black!00, inner sep=2pt, minimum width=2pt] (y12) at (0,-6) {\scriptsize $y_{ije}$};
        \node[circle, draw = black!100, fill=black!00, inner sep=2pt, minimum width=2pt] (y21) at (8,-6) {\scriptsize $y_{jie}$};
        \node[circle, draw = black!100, fill=black!00, inner sep=2pt, minimum width=2pt] (t12) at (4,-8) {\scriptsize $z_{e}$};

        \draw[-]    (y12) edge (y21) ;
        \draw[-]    (z12) edge (y21) ;
        \draw[-]    (z12) edge (y12) ;
        \draw[-]    (z12) edge (t12) ;
        \draw[-]    (z2) edge (y12) ;
        \draw[-]    (z1) edge (y21) ;
	  	
    \end{tikzpicture}
\end{center}

To represent the facets of the triangulation, we introduce a symbol corresponding to each variable: Let $i\in V$ and $e=ij\in E$:
\begin{itemize}
    \item the variable $z_i$ is represented by the symbol $\circ$.  The vertex $i$ is instead represented by $\bullet$ if $z_i$ is not present. 
    \item the variable $z_{e}$ is represented by the edge type $-$,
    \item the variable $t_{e}$ is represented by the edge type $\middlewave{0.5cm}$,
    \item the variable $y_{ije}$ is represented by the edge type $\rightarrow$ pointing from $i$ to $j$, and 
    \item the variable $y_{jie}$ is represented by the edge type $\leftarrow$ pointing from $j$ to $i$. 
\end{itemize}
Given a subset $S$ of the generators of $R_{G}$, we let $G_S$ denote the graph drawn with the symbols above according to the elements in $S$.  
We also let $Z \coloneqq= \{z_i~ :~ i\in V\}$, $Z_S \coloneqq S\cap Z$ and $\mathfrak{Z}_S \coloneqq \{ i \in V~ : ~z_i \in Z_S\}$.
For example, if $G$ is a path on $3$ vertices, we represent the set of variables $S = \{z_1, z_3, y_{12}, z_{12}, t_{23}\}$ via the graph $G_S$:
\begin{center}
\bigskip

    \begin{tikzpicture}[thick,scale=0.5]
        \tikzset{decoration={snake,amplitude=.4mm,segment length=2mm, post length=0mm,pre length=0mm}}
	
        \node[circle, draw = black!100, fill=black!00, inner sep=2pt, minimum width=2pt] (1) at (0,0) {};
        \node[circle, draw = black!100, fill=black!100, inner sep=2pt, minimum width=2pt] (2) at (2,0) {};
        \node[circle, draw = black!100, fill=black!00, inner sep=2pt, minimum width=2pt] (3) at (4,0) {};

	\draw[->]   (1) edge[bend left] (2) ;
        \draw[-]    (1) edge (2) ;
        \draw[decorate] (2) -- (3) ; 
	 
	\node at (0,-1) {\footnotesize 1} ;
	\node at (2,-1) {\footnotesize 2} ;
	\node at (4,-1) {\footnotesize 3} ;
	  	
    \end{tikzpicture}
\bigskip

\end{center}
For this example, $Z = \{z_1,z_2,z_3\}$, $Z_S = \{z_1, z_3\}$ and $\mathfrak{Z}_S = \{1,3\}$. 

The fundamental binomials imply that if a subset of variables corresponds to a face of the triangulation, then its associated graph does not contain any of the following subgraphs:

\begin{equation}
\label{eqn: fundamental subgraphs}
    \begin{tikzpicture}[thick,scale=0.5]
        \tikzset{decoration={snake,amplitude=.4mm,segment length=2mm, post length=0mm,pre length=0mm}}
	
        \node[circle, draw = black!00, fill=black!00, inner sep=2pt, minimum width=2pt] (1) at (0,0) {};
        \node[circle, draw = black!00, fill=black!00, inner sep=2pt, minimum width=2pt] (2) at (2,0) {};

        \node[circle, draw = black!00, fill=black!00, inner sep=2pt, minimum width=2pt] (3) at (0+4,0) {};
        \node[circle, draw = black!00, fill=black!00, inner sep=2pt, minimum width=2pt] (4) at (2+4,0) {};

        \node[circle, draw = black!00, fill=black!00, inner sep=2pt, minimum width=2pt] (5) at (0+8,0) {};
        \node[circle, draw = black!00, fill=black!00, inner sep=2pt, minimum width=2pt] (6) at (2+8,0) {};

        \node[circle, draw = black!00, fill=black!00, inner sep=2pt, minimum width=2pt] (7) at (0+12,0) {};
        \node[circle, draw = black!00, fill=black!00, inner sep=2pt, minimum width=2pt] (8) at (2+12,0) {};

        \node[circle, draw = black!100, fill=black!00, inner sep=2pt, minimum width=2pt] (9) at (0+16,0) {};
        \node[circle, draw = black!00, fill=black!00, inner sep=2pt, minimum width=2pt] (10) at (2+16,0) {};

        \node[circle, draw = black!00, fill=black!00, inner sep=2pt, minimum width=2pt] (11) at (0+20,0) {};
        \node[circle, draw = black!100, fill=black!00, inner sep=2pt, minimum width=2pt] (12) at (2+20,0) {};

	\draw[-]   (1) edge[bend left] (2) ;
        \draw[decorate]    (1) -- (2) ;

        \draw[->]   (3) edge[bend left] (4) ;
        \draw[decorate]    (3) -- (4) ;

        \draw[<-]   (5) edge[bend left] (6) ;
        \draw[decorate]    (5) -- (6) ;

        \draw[->]   (7) edge[bend left] (8) ;
        \draw[<-]    (7) edge[bend right] (8) ;

        \draw[<-]   (9) edge (10) ;

        \draw[->]   (11) edge (12) ;
	 
	  	
    \end{tikzpicture}
\end{equation}
We refer to these six subgraph as \emph{fundamental obstructions}. 
The following lemma is immediate.
\begin{lemma}
    \label{lem: only double}
    Let $G$ be a simple, connected and undirected graph, and let $\mathcal{T}$ be a regular unimodular triangulation of $\mathcal{C}_G$ given by a good term order on $R_G$.  
    If $S$ is a facet of $\mathcal{T}$, then $G_S$ contains only single edges and double edges. 
    Moreover, any double edges are of the form 
    \begin{center}
    \begin{tikzpicture}[thick,scale=0.5]
        \tikzset{decoration={snake,amplitude=.4mm,segment length=2mm, post length=0mm,pre length=0mm}}
	
        \node[circle, draw = black!00, fill=black!00, inner sep=2pt, minimum width=2pt] (1) at (0,0) {};
        \node[circle, draw = black!00, fill=black!00, inner sep=2pt, minimum width=2pt] (2) at (2,0) {};

        \node[circle, draw = black!00, fill=black!00, inner sep=2pt, minimum width=2pt] (3) at (10,0) {};
        \node[circle, draw = black!00, fill=black!00, inner sep=2pt, minimum width=2pt] (4) at (12,0) {};

	\draw[->]   (1) edge[bend left] (2) ;
        \draw[-]    (1) -- (2) ;

        \draw[<-]   (3) edge[bend left] (4) ;
        \draw[-]    (3) -- (4) ;

        \node at (6,0) {or} ;
    \end{tikzpicture}.
    \end{center}
\end{lemma}

Given a subset $S$ of the variables in $R_G$, we define the \emph{support graph} of $S$ (or $G_S$) to be the graph on vertex set $V$ and edge set 
\[
\{e=ij\in E~:~S\cap\{t_e,z_e,y_{ije},y_{jie}\}\neq \emptyset\}.
\]

The following statement shows that the support graph of any facet is as large as possible.

\begin{proposition}
\label{prop: connected}
Let $G=(V,E)$ be a connected, undirected graph and let $\mathcal{T}$ be a triangulation of $\mathcal{C}_G$ coming from a good term order. Let $S$ be a subset of the variables of $R_G$. If $S$ is a facet of $\mathcal{T}$, then the support graph of $S$ equals $G$. In particular, the support graph of $S$ is connected.
\end{proposition}

\begin{proof}
    Let $S$ be a facet of $\mathcal{T}$ and assume by contradiction that there is some edge $e=ij\in E$ such that $S\cap\{t_e,z_e,y_{ije},y_{jie}\}=\emptyset$. We note that the variable $t_e$ does not appear in any of the zig-zag binomials, the cyclic binomials nor the cyclic binomials. The only occurrence of $t_e$ is in the leading term of the fundamental binomials $y_{ije}t_e-z_i^2$, $y_{jie}t_e-z_j^2$ and $t_ez_e-z_iz_j$. However, since by assumption, none of $y_{ije}$, $y_{jie}$ and $z_e$ is contained in $S$, it follows that $S\cup\{t_e\}$ is also a face of $\mathcal{T}$. This contradicts the fact that $S$ is a facet.
\end{proof}

In the coming sections we apply these results to derive explict characterizations of the facets of regular unimodular triangulations of $\mathcal{C}_G$ arising from good term orders on $R_G$ for special instances of $G$. 
In Section~\ref{sec: the path}, we characterize the facets of this triangulation for a specific good term order when $G$ is the path graph. 
In Section~\ref{sec: the cycle}, we show that the techniques in Section~\ref{sec: the path} can be extended to yield an analogous characterization of the facets of a triangulation for the cycle. 
Finally, in Section~\ref{sec: trees}, we extend the characterization of the facets of the triangulation for paths to general trees. 

\section{The Cosmological Polytope of the Path}
\label{sec: the path}
In this section, we give an explicit description of the regular unimodular triangulation corresponding to a Gr\"obner basis with respect to a good term order of the toric ideal for the cosmological polytope of the \emph{n-path}, $I_n$; that is, the graph with vertex set $V =[n+1]$ and edge set
$
E = \{ii+1~ :~ i\in[n]\}.
$

A combinatorial description of the facets of this polytope is given that allows for enumeration of the facets.
The resulting formula for the normalized volume of $\mathcal{C}_{I_n}$ agrees with the formula identified in \cite{kuhne2022faces}.
The combinatorial description of the facets may also be used to compute the canonical form of the polytope in a novel way, which may suggest new physical theories for the computation of wavefunctions associated to such Feynman diagrams.

In the following, we use the variable order
\begin{equation}
\begin{split}
\label{eqn: path variable order}
&y_{12} > y_{23} > \cdots > y_{n-1n} > y_{nn-1} > \cdots > y_{32} > y_{21} > z_{12} > \cdots\\
&\cdots> z_{n-1n} > t_{12} > \cdots > t_{n-1n} > z_1 > \cdots > z_n,
\end{split}
\end{equation}
where for the edge $e=ii+1$, we write $y_{ii+1}$ and $y_{i+1i}$ for the variables $y_{ii+1e}$ and $y_{i+1ie}$, respectively. 
It can be checked that the lexicographic term order, with respect to this ordering of the variables, on the monomials in $R_{I_n}$ is a good term order according to \Cref{def : good}.  
Since the cosmological polytope $\mathcal{C}_{G}$ for a graph $G = (V, E)$ has dimension $|V| + |E| -1$, the corresponding regular unimodular triangulation has facets given by all $(2n+1)$-subsets of the variables $y_{ije}, y_{jie}, t_e, z_e, z_i$ that do not contain the leading terms of the binomials in this Gr\"obner basis. Our goal is to characterize these subsets $S$ in terms of the structure of their graphs $G_S$ defined in Subsection~\ref{subsec: reg tri}.

By Proposition~\ref{prop: connected}, we know that $G_S$ is connected whenever $S$ is a facet.  
We also know from Lemma~\ref{lem: only double} that all edges in $G_S$ are either single or double edges and all double edges are of the form
\begin{center}
    \begin{tikzpicture}[thick,scale=0.5]
        \tikzset{decoration={snake,amplitude=.4mm,segment length=2mm, post length=0mm,pre length=0mm}}
	
        \node[circle, draw = black!00, fill=black!00, inner sep=2pt, minimum width=2pt] (1) at (0,0) {};
        \node[circle, draw = black!00, fill=black!00, inner sep=2pt, minimum width=2pt] (2) at (2,0) {};

        \node[circle, draw = black!00, fill=black!00, inner sep=2pt, minimum width=2pt] (3) at (10,0) {};
        \node[circle, draw = black!00, fill=black!00, inner sep=2pt, minimum width=2pt] (4) at (12,0) {};

	\draw[->]   (1) edge[bend left] (2) ;
        \draw[-]    (1) -- (2) ;

        \draw[<-]   (3) edge[bend left] (4) ;
        \draw[-]    (3) -- (4) ;

        \node at (6,0) {or} ;
    \end{tikzpicture}
    \end{center}
Since the path graph contains no cycles, the Gr\"obner basis given in Theorem~\ref{thm: gb} for $I_{I_{n}}$ contains only the fundamental binomials and zig-zag binomials. 
Now that we have specified a specific good term order on $R_{I_{n}}$ we can also identify the subgraphs forbidden by the leading terms of the zig-zag binomials.

Namely, if $S$ is a facet of the triangulation then $G_S$ does not contain any partially directed paths oriented to the right that end with a $\circ$; that is, it does not contain any subgraphs of the following form:
Let $\pi = \{ii+1, i+1i+2,\ldots, j-1j\}$ with $1\leq i<j\leq n+1$ be a subpath of $I_n$.
Given a partition $(E_1,E_2)$ of the edges of $\pi$ with $E_1\neq \emptyset$, we call the graph $G_R$ for the set of symbols 
\[
R = \{y_{\ell\ell+1} ~:~ \ell\ell+1 \in E_1\}\cup\{z_{\ell\ell+1}~ :~ \ell\ell+1 \in E_2\} \cup \{z_j\}
\]
a \emph{partially directed path to the right (ending in $\circ$)}. 
For example, if $S$ is a facet it cannot contain the subset of symbols $R$ yielding the following graph:

\begin{center}
\bigskip

    \begin{tikzpicture}[thick,scale=0.5]
        \tikzset{decoration={snake,amplitude=.4mm,segment length=2mm, post length=0mm,pre length=0mm}}
	
        \node[circle, draw = black!100, fill=black!100, inner sep=2pt, minimum width=2pt] (1) at (0,0) {};
        \node[circle, draw = black!100, fill=black!100, inner sep=2pt, minimum width=2pt] (2) at (2,0) {};

        \node[circle, draw = black!100, fill=black!100, inner sep=2pt, minimum width=2pt] (3) at (0+4,0) {};
        \node[circle, draw = black!100, fill=black!100, inner sep=2pt, minimum width=2pt] (4) at (2+4,0) {};

        \node[circle, draw = black!100, fill=black!100, inner sep=2pt, minimum width=2pt] (5) at (0+8,0) {};
        \node[circle, draw = black!100, fill=black!100, inner sep=2pt, minimum width=2pt] (6) at (2+8,0) {};

        \node[circle, draw = black!100, fill=black!00, inner sep=2pt, minimum width=2pt] (7) at (0+12,0) {};

	\draw[-]    (1) -- (2) ;
        \draw[->]    (2) -- (3) ;
        \draw[->]    (3) -- (4) ;
        \draw[-]    (4) -- (5) ;
        \draw[->]    (5) -- (6) ;
        \draw[-]    (6) -- (7) ;
	  	
    \end{tikzpicture}
\bigskip

\end{center}

The following lemma collects some additional useful properties of $G_S$ when $S$ is a facet.  
\begin{lemma}
\label{lem: facet means connected}
Let $S$ be a subset of the variables generating the ring $R_{I_n}$.  If $S$ is a facet of the triangulation and 
$
\mathfrak{Z}_S = \{i_1< i_2 <\cdots < i_{n+1-k}\}, 
$
it follows that
\begin{enumerate}
    \item $G_S$ contains exactly $k$ double edges,
	




    \item the number of double edges in the induced subgraph of $G_S$ on $[i_1]$ is $i_1 - 1$,
    \item the number of double edges in the induced subgraph of $G_S$ on $\{i_j,\ldots, i_{j+1}\}$ is $i_{j+1} - i_j - 1$, for all $j\in[n-k]$, and
    \item the number of double edges in the induced subgraph of $G_S$ on $\{i_{n-k+1},\ldots, n+1\}$ is $n+1 - i_{n-k+1}$.
\end{enumerate}
\end{lemma}
\begin{proof}

Since $S$ is a facet we know that $|S| = 2n+1$.  
We know also from Proposition~\ref{prop: connected} that the support graph of $G_S$ is connected and equal to $I_{n}$. 
Hence, there is at least one edge in $G_S$ for all $n$ edges in $I_{n}$. 
Moreover, by Lemma~\ref{lem: only double} we know that $G_S$ only contains single and double edges. 
Since $|Z_S| = n+1-k$ and $|S| = 2n+1$, it follows that $G_S$ contains exactly $k$ double edges. 

Consider now the subgraph between $i < j$ of $G_S$ where $z_i, z_j \in S$ and $z_\ell\notin S$ for all $i < \ell < j$.  
We claim that there are at most $j - i - 1$ double edges in this subgraph.  
To see this, suppose there are $j-i$ double edges instead.  
It follows that all edges in this subgraph are double and of the form specified in Lemma~\ref{lem: only double}. 
Since the subgraph cannot include the fundamental obstruction $\circ\leftarrow$, it follows that the first pair of double edges is of the form
\begin{center}

    \begin{tikzpicture}[thick,scale=0.5]
        \tikzset{decoration={snake,amplitude=.4mm,segment length=2mm, post length=0mm,pre length=0mm}}
	
        \node[circle, draw = black!100, fill=black!00, inner sep=2pt, minimum width=2pt] (1) at (0,0) {};
        \node[circle, draw = black!100, fill=black!100, inner sep=2pt, minimum width=2pt] (2) at (2,0) {};

	\draw[->]   (1) edge[bend left] (2) ;
        \draw[-]    (1) edge (2) ;
	 
	\node at (0,-1) {\footnotesize $i$} ;
	\node at (2,-1) {\footnotesize $i+1$} ;
	  	
    \end{tikzpicture}

\end{center}
Since all remaining edges in the subgraph must also be double edges, and since these sets of doubles must each include the undirected edge $\ell\ell+1$ (for $i\leq \ell\leq j-1$), it follows that the subgraph contains a partially directed path to the right ending in $\circ$, which is a forbidden subgraph by the leading term of some zig-zag binomial. 
Hence, we have a contradiction.  

It then follows from the Pigeonhole Principle that each subgraph of $G_S$ given by a pair of nodes $i<j$ for $z_i,z_j\in S$ but $z_\ell\notin S$ for all $i<\ell<k$, or $z_i\in S$ but $z_\ell\notin S$ for all $\ell < i$, or $z_i \in S$ but $z_\ell\notin S$ for all $i <\ell$ contains exactly as many double edges as it does black nodes. 
This proves claims (2) -- (4).
\end{proof}

Based on Lemma~\ref{lem: facet means connected}, it can be helpful to consider facets according to their intersection with the set $Z = \{z_i~ :~ i\in[n+1]\}$.  
If a facet $S$ is such that $Z_S = S\cap Z = \{z_{i_1},\ldots, z_{i_k}\}$, where $i_1<i_2<\cdots<i_k$, we can partition the graph $G_S$ into the induced subgraphs on node sets $\{1,\ldots, i_1\}$, $\{i_k,\ldots, n+1\}$ and $\{i_j, i_j+1,\ldots, i_{j+1}\}$ for all $j\in[k-1]$, and consider the possible placements of the appropriate number of edges in each induced subgraph so as to ensure that $|S| = 2n+1$.  
A rule for producing all such graphs in this way will yield a combinatorial description of the facets of the triangulation. 
The next theorem gives such a characterization of the graphs that correspond to facets of the triangulation.
In the following we use $\leftrightarrow$ to denote that we are free to choose between either arrow (either $\leftarrow$ or $\rightarrow$). 
\begin{theorem}
\label{lem: path tri facets}
Let $S$ be a subset of the generators of the ring $R_{I_{n}}$ and let $Z_S = \{z_{i_1},\ldots, z_{i_k}\}$ where $i_1<\cdots < i_k$.  
Then $S$ is a facet of the triangulation of $\mathcal{C}_{I_{n}}$ corresponding to the lexicographic order induced by \eqref{eqn: path variable order} if and only if all three of the following hold:
    \begin{enumerate}
        \item The induced subgraph of $G_S$ on nodes $[i_1]$ is of the form
        \begin{center}
        \medskip
        
        \begin{tikzpicture}[thick,scale=0.5]
        \tikzset{decoration={snake,amplitude=.4mm,segment length=2mm, post length=0mm,pre length=0mm}}
	
        \node[circle, draw = black!100, fill=black!100, inner sep=2pt, minimum width=2pt] (1) at (0,0) {};
        \node[circle, draw = black!100, fill=black!100, inner sep=2pt, minimum width=2pt] (2) at (2,0) {};

        \node[circle, draw = black!100, fill=black!100, inner sep=2pt, minimum width=2pt] (3) at (0+4,0) {};
        \node[circle, draw = black!100, fill=black!100, inner sep=2pt, minimum width=2pt] (4) at (2+4,0) {};

        \node[circle, draw = black!100, fill=black!100, inner sep=2pt, minimum width=2pt] (5) at (0+8,0) {};
        \node[circle, draw = black!100, fill=black!100, inner sep=2pt, minimum width=2pt] (6) at (2+8,0) {};

        \node[circle, draw = black!100, fill=black!100, inner sep=2pt, minimum width=2pt] (7) at (0+12,0) {};
        \node[circle, draw = black!100, fill=black!00, inner sep=2pt, minimum width=2pt] (8) at (2+12,0) {};
        
	\draw[-]    (1) -- (2) ;
        \draw[-]    (2) -- (3) ;
        \draw[dotted]    (3) -- (4) ;
        \draw[-]    (4) -- (5) ;
        \draw[-]    (5) -- (6) ;
        \draw[-]    (6) -- (7) ;
        \draw[-]    (7) -- (8) ;

        \draw[<-]    (1) edge[bend left] (2) ;
        \draw[<-]    (2) edge[bend left] (3) ;
	  \draw[<-]    (4) edge[bend left] (5) ;
        \draw[<-]    (5) edge[bend left] (6) ;
        \draw[<-]    (6) edge[bend left] (7) ;
        \draw[<-]    (7) edge[bend left] (8) ;
        
        \end{tikzpicture}.
        \medskip
        
        \end{center}
        That is, all edges are double with a $\leftarrow$.
        
        \item For all $j\in[k-1]$, the induced subgraph of $G_S$ on $\{i_j,i_j+1,\ldots, i_{j+1}\}$ is of the form
        \begin{center}
        \medskip
        
        \begin{tikzpicture}[thick,scale=0.5]
        \tikzset{decoration={snake,amplitude=.4mm,segment length=2mm, post length=0mm,pre length=0mm}}
	
        \node[circle, draw = black!100, fill=black!00, inner sep=2pt, minimum width=2pt] (1) at (0,0) {};
        \node[circle, draw = black!100, fill=black!100, inner sep=2pt, minimum width=2pt] (2) at (2,0) {};

        \node[circle, draw = black!100, fill=black!100, inner sep=2pt, minimum width=2pt] (3) at (0+4,0) {};
        \node[circle, draw = black!100, fill=black!100, inner sep=2pt, minimum width=2pt] (4) at (2+4,0) {};

        \node[circle, draw = black!100, fill=black!100, inner sep=2pt, minimum width=2pt] (5) at (0+8,0) {};
        \node[circle, draw = black!100, fill=black!100, inner sep=2pt, minimum width=2pt] (6) at (2+8,0) {};

        \node[circle, draw = black!100, fill=black!100, inner sep=2pt, minimum width=2pt] (7) at (0+12,0) {};
        \node[circle, draw = black!100, fill=black!00, inner sep=2pt, minimum width=2pt] (8) at (2+12,0) {};

	\draw[-]    (1) -- (2) ;
        \draw[-]    (2) -- (3) ;
        \draw[dotted]    (3) -- (4) ;
        \draw[-]    (4) -- (5) ;
        \draw[decorate]    (5) -- (6) ;
        \draw[-]    (6) -- (7) ;
        \draw[-]    (7) -- (8) ;

        \draw[->]    (1) edge[bend left] (2) ;
        \draw[<->]    (2) edge[bend left] (3) ;
	  \draw[<->]    (4) edge[bend left] (5) ;
        \draw[<-]    (6) edge[bend left] (7) ;
        \draw[<-]    (7) edge[bend left] (8) ;
        
        \end{tikzpicture}
        \medskip
        
        \end{center}
        or
        \begin{center}
        \medskip
        
        \begin{tikzpicture}[thick,scale=0.5]
        \tikzset{decoration={snake,amplitude=.4mm,segment length=2mm, post length=0mm,pre length=0mm}}
	
        \node[circle, draw = black!100, fill=black!00, inner sep=2pt, minimum width=2pt] (1) at (0,0) {};
        \node[circle, draw = black!100, fill=black!100, inner sep=2pt, minimum width=2pt] (2) at (2,0) {};

        \node[circle, draw = black!100, fill=black!100, inner sep=2pt, minimum width=2pt] (3) at (0+4,0) {};
        \node[circle, draw = black!100, fill=black!100, inner sep=2pt, minimum width=2pt] (4) at (2+4,0) {};

        \node[circle, draw = black!100, fill=black!100, inner sep=2pt, minimum width=2pt] (5) at (0+8,0) {};
        \node[circle, draw = black!100, fill=black!100, inner sep=2pt, minimum width=2pt] (6) at (2+8,0) {};

        \node[circle, draw = black!100, fill=black!100, inner sep=2pt, minimum width=2pt] (7) at (0+12,0) {};
        \node[circle, draw = black!100, fill=black!00, inner sep=2pt, minimum width=2pt] (8) at (2+12,0) {};

	\draw[-]    (1) -- (2) ;
        \draw[-]    (2) -- (3) ;
        \draw[dotted]    (3) -- (4) ;
        \draw[-]    (4) -- (5) ;
        \draw[<-]    (5) -- (6) ;
        \draw[-]    (6) -- (7) ;
        \draw[-]    (7) -- (8) ;

        \draw[->]    (1) edge[bend left] (2) ;
        \draw[<->]    (2) edge[bend left] (3) ;
	  \draw[<->]    (4) edge[bend left] (5) ;
        \draw[<-]    (6) edge[bend left] (7) ;
        \draw[<-]    (7) edge[bend left] (8) ;
        
        \end{tikzpicture}
        \medskip
        
        \end{center}
        or 
        \begin{center}
        \medskip
        
        \begin{tikzpicture}[thick,scale=0.5]
        \tikzset{decoration={snake,amplitude=.4mm,segment length=2mm, post length=0mm,pre length=0mm}}
	
        \node[circle, draw = black!100, fill=black!00, inner sep=2pt, minimum width=2pt] (1) at (0,0) {};
        \node[circle, draw = black!100, fill=black!100, inner sep=2pt, minimum width=2pt] (2) at (2,0) {};

        \node[circle, draw = black!100, fill=black!100, inner sep=2pt, minimum width=2pt] (3) at (0+4,0) {};
        \node[circle, draw = black!100, fill=black!100, inner sep=2pt, minimum width=2pt] (4) at (2+4,0) {};

        \node[circle, draw = black!100, fill=black!100, inner sep=2pt, minimum width=2pt] (5) at (0+8,0) {};
        \node[circle, draw = black!100, fill=black!100, inner sep=2pt, minimum width=2pt] (6) at (2+8,0) {};

        \node[circle, draw = black!100, fill=black!100, inner sep=2pt, minimum width=2pt] (7) at (0+12,0) {};
        \node[circle, draw = black!100, fill=black!00, inner sep=2pt, minimum width=2pt] (8) at (2+12,0) {};

	\draw[decorate]    (1) -- (2) ;
        \draw[-]    (2) -- (3) ;
        \draw[dotted]    (3) -- (4) ;
        \draw[-]    (4) -- (5) ;
        \draw[-]    (5) -- (6) ;
        \draw[-]    (6) -- (7) ;
        \draw[-]    (7) -- (8) ;

        \draw[<-]    (2) edge[bend left] (3) ;
	  \draw[<-]    (4) edge[bend left] (5) ;
        \draw[<-]    (5) edge[bend left] (6) ;
        \draw[<-]    (6) edge[bend left] (7) ;
        \draw[<-]    (7) edge[bend left] (8) ;
        
        \end{tikzpicture}
        \medskip
        
        \end{center}
        or
        \begin{center}
        \medskip
        
        \begin{tikzpicture}[thick,scale=0.5]
        \tikzset{decoration={snake,amplitude=.4mm,segment length=2mm, post length=0mm,pre length=0mm}}
	
        \node[circle, draw = black!100, fill=black!00, inner sep=2pt, minimum width=2pt] (1) at (0,0) {};
        \node[circle, draw = black!100, fill=black!100, inner sep=2pt, minimum width=2pt] (2) at (2,0) {};

        \node[circle, draw = black!100, fill=black!100, inner sep=2pt, minimum width=2pt] (3) at (0+4,0) {};
        \node[circle, draw = black!100, fill=black!100, inner sep=2pt, minimum width=2pt] (4) at (2+4,0) {};

        \node[circle, draw = black!100, fill=black!100, inner sep=2pt, minimum width=2pt] (5) at (0+8,0) {};
        \node[circle, draw = black!100, fill=black!100, inner sep=2pt, minimum width=2pt] (6) at (2+8,0) {};

        \node[circle, draw = black!100, fill=black!100, inner sep=2pt, minimum width=2pt] (7) at (0+12,0) {};
        \node[circle, draw = black!100, fill=black!00, inner sep=2pt, minimum width=2pt] (8) at (2+12,0) {};

	\draw[-]    (1) -- (2) ;
        \draw[-]    (2) -- (3) ;
        \draw[dotted]    (3) -- (4) ;
        \draw[-]    (4) -- (5) ;
        \draw[-]    (5) -- (6) ;
        \draw[-]    (6) -- (7) ;
        \draw[-]    (7) -- (8) ;

        \draw[<-]    (2) edge[bend left] (3) ;
	  \draw[<-]    (4) edge[bend left] (5) ;
        \draw[<-]    (5) edge[bend left] (6) ;
        \draw[<-]    (6) edge[bend left] (7) ;
        \draw[<-]    (7) edge[bend left] (8) ;
        
        \end{tikzpicture}.
        \medskip
        
        \end{center}
        That is, either (1) exactly one edge whose least vertex is a black node
        is either $\middlewave{0.5cm}$ or $\leftarrow$, all edges to the right of this edge are double with a $\leftarrow$ and all edges to the left of this edge are double with either arrow ($\leftarrow$ or $\rightarrow$), except for the first edge which must have $\rightarrow$, or (2) the leftmost edge is either $-$ or $\middlewave{0.5cm}$ and all edges to the right are double with a $\leftarrow$. 
        
        \item The induced subgraph of $G_S$ on nodes $\{i_k, i_k+1,\ldots, n+1\}$ is of the form
        \begin{center}
        \medskip
        
        \begin{tikzpicture}[thick,scale=0.5]
        \tikzset{decoration={snake,amplitude=.4mm,segment length=2mm, post length=0mm,pre length=0mm}}
	
        \node[circle, draw = black!100, fill=black!00, inner sep=2pt, minimum width=2pt] (1) at (0,0) {};
        \node[circle, draw = black!100, fill=black!100, inner sep=2pt, minimum width=2pt] (2) at (2,0) {};

        \node[circle, draw = black!100, fill=black!100, inner sep=2pt, minimum width=2pt] (3) at (0+4,0) {};
        \node[circle, draw = black!100, fill=black!100, inner sep=2pt, minimum width=2pt] (4) at (2+4,0) {};

        \node[circle, draw = black!100, fill=black!100, inner sep=2pt, minimum width=2pt] (5) at (0+8,0) {};
        \node[circle, draw = black!100, fill=black!100, inner sep=2pt, minimum width=2pt] (6) at (2+8,0) {};

        \node[circle, draw = black!100, fill=black!100, inner sep=2pt, minimum width=2pt] (7) at (0+12,0) {};
        \node[circle, draw = black!100, fill=black!100, inner sep=2pt, minimum width=2pt] (8) at (2+12,0) {};

	\draw[-]    (1) -- (2) ;
        \draw[-]    (2) -- (3) ;
        \draw[dotted]    (3) -- (4) ;
        \draw[-]    (4) -- (5) ;
        \draw[-]    (5) -- (6) ;
        \draw[-]    (6) -- (7) ;
        \draw[-]    (7) -- (8) ;

        \draw[->]    (1) edge[bend left] (2) ;
        \draw[<->]    (2) edge[bend left] (3) ;
	  \draw[<->]    (4) edge[bend left] (5) ;
        \draw[<->]    (5) edge[bend left] (6) ;
        \draw[<->]    (6) edge[bend left] (7) ;
        \draw[<->]    (7) edge[bend left] (8) ;
        
        \end{tikzpicture}.
        \medskip
        
        \end{center}
        That is, all edges are double with either arrow (either $\leftarrow$ or $\rightarrow$), except for the first edge which must have a $\rightarrow$. 
    \end{enumerate}
\end{theorem}

\begin{proof}
We first observe that any set $S$ such that $G_S$ satisfies the listed properties, is a facet. 
Notice first that any choice of the edges for each of the possible subgraphs does not contain an induced subgraph excluded by the fundamental binomials. 
Furthermore, any partially directed path to the right is either interrupted by a single edge of the form $\leftarrow$ or $\middlewave{0.5cm}$, or it terminates in a black node. 
Hence, such a $G_S$ also does not contain any subgraph forbidden by the leading terms of the zig-zag binomials.
Since there is exactly one double edge for every black node, it also follows that $|S| = 2n+1$.  
Since the dimension of $\mathcal{C}_{I_{n}}$ is $2n$, it follows that $G_S$ is a facet of the triangulation. 

Suppose now that $S$ is a facet of the triangulation, and consider its associated graph $G_S$.  
Since $S$ is a facet, we know $|S| = 2n+1$, and by \Cref{lem: facet means connected} we also know that $G_S$ is connected and any of the induced subgraphs on node sets $\{1,\ldots, i_1\}$, $\{i_k,\ldots, n+1\}$ and $\{i_j, i_j+1,\ldots, i_{j+1}\}$ for $j\in[k-1]$ contains as many black nodes as it does double edges.  It therefore suffices to show that these subgraphs of $G_S$ are of  one of the possible forms specified in the above list. 

Consider first the induced subgraph of $G_S$ on node set $[i_1]$. 
Since $S$ is a facet, by Lemma~\ref{lem: facet means connected}, we know that every edge in this subgraph is a double edge, and hence of the form $\{(i,i+1), ii+1\}$ or $\{(i+1,i), ii+1\}$. 
Since $S$ cannot contain the leading term of any fundamental binomial, it does not contain both $z_{i_1}$ and $y_{i_1-1i_1}$.  
Hence, this subgraph must contain the double edge $\{(i_1,i_1-1), i_1-1i_1\}$. 
Similarly, since $S$ cannot contain the leading term of any zig-zag binomial, this subgraph cannot contain any partially directed paths to the right.  
It follows that all double edges in this subgraph are of the form $\{(i+1,i), ii+1\}$. 
Hence, $G_S$ fulfills the first criterion in the above list.

Similarly, for the induced graph of $G_S$ on node set $\{i_j, i_j+1,\ldots, i_{j+1}\}$, we know that the graph must be connected and contain exactly $i_{j+1} - i_j -1$ double edges by Lemma~\ref{lem: facet means connected}. 
Hence, there is exactly one single edge in the graph. 
Suppose that this edge is the leftmost edge (i.e., between $i_j$ and $i_{j}+1$). 
In this case, the edge may be either $\middlewave{0.5cm}$ or $-$, but not $\leftarrow$ or $\rightarrow$.
To see that it cannot be $\leftarrow$, note that this would mean that the leading term of a fundamental binomial is contained in $S$.  
To see that it cannot be $\rightarrow$, note that, since all remaining edges in the subgraph must be doubled (and hence include a $-$), it would follow that $S$ contains the leading term of a zig-zag binomial, which is a contradiction. 
In a similar fashion, all double edges must be of the form $\{(i+1,i), ii+1\}$. 
Otherwise $S$ would contain the leading term of a zig-zag binomial. 

Suppose now that the single edge in the subgraph is between $i_j+t$ and $i_j+{t+1}$ for some $t>1$.  
By the same argument as the previous case, all remaining edges must be double edges and all double edges to the right of $i_j+{t+1}$ must be of the form $\{(i+1,i), ii+1\}$.  
We must also have that the double edge between $i_j$ and $i_j+1$ is of the form $\{(i_j,i_j+1), i_ji_j+1\}$, 
since otherwise $S$ would contain the leading term of a fundamental binomial. 
However, all double edges between $i_j+s$ and $i_j+s+1$ for $1\leq s< t$ can be of either form $\{(i+1,i), ii+1\}$ or $\{(i,i+1), ii+1\}$, 
since the single edge will interrupt any partially directed path to the right. 
Observe further that the single edge must be of the form $\leftarrow$ or $\middlewave{0.5cm}$, since any other option would combine with the undirected edges and the directed edge between $i_j$ and $i_{j}+1$ to yield a partially directed path to the right terminating in a $\circ$. 
It follows that if $S$ is a facet, the corresponding induced subgraphs of $G_S$ on the intervals $\{i_j,i_j+1,\ldots, i_{j+1}\}$ for all $j\in[k-1]$ are of the form in item (2) in the above list. 

Finally, for the induced subgraph of $G_S$ on node set $\{i_k,\ldots, n+1\}$, we know from Lemma~\ref{lem: facet means connected} that all edges are double edges and hence of the form $\{(i+1,i), ii+1\}$ or $\{(i,i+1), ii+1\}$.  
To avoid a subgraph forbidden by a fundamental binomial, we must also have that the double edge between $i_k$ and $i_k+1$ is of the form $\{(i,i+1), ii+1\}$.  
However, since the path does not contain any $\circ$ to the right of node $i_k$, we are free to choose the direction of the arrow in all remaining double edges.  
Hence, this subgraph is of the form given in item (3) in the above list, which completes the proof.
\end{proof}

\begin{example}
\label{ex: facets of I_3}
According to Theorem~\ref{lem: path tri facets}, the facets of the triangulation of $\mathcal{C}_{I_2}$ are given by the following sixteen graphs:
\begin{center}
        \medskip
        
        \begin{tikzpicture}[thick,scale=0.5]
        \tikzset{decoration={snake,amplitude=.4mm,segment length=2mm, post length=0mm,pre length=0mm}}
	
        \node[circle, draw = black!100, fill=black!00, inner sep=2pt, minimum width=2pt] (1) at (0,0) {};
        \node[circle, draw = black!100, fill=black!00, inner sep=2pt, minimum width=2pt] (2) at (2,0) {};
        \node[circle, draw = black!100, fill=black!00, inner sep=2pt, minimum width=2pt] (3) at (4,0) {};

        \node[circle, draw = black!100, fill=black!00, inner sep=2pt, minimum width=2pt] (4) at (0+8,0) {};
        \node[circle, draw = black!100, fill=black!00, inner sep=2pt, minimum width=2pt] (5) at (2+8,0) {};
        \node[circle, draw = black!100, fill=black!00, inner sep=2pt, minimum width=2pt] (6) at (4+8,0) {};

        \node[circle, draw = black!100, fill=black!00, inner sep=2pt, minimum width=2pt] (7) at (0+16,0) {};
        \node[circle, draw = black!100, fill=black!00, inner sep=2pt, minimum width=2pt] (8) at (2+16,0) {};
        \node[circle, draw = black!100, fill=black!00, inner sep=2pt, minimum width=2pt] (9) at (4+16,0) {};

        \node[circle, draw = black!100, fill=black!00, inner sep=2pt, minimum width=2pt] (10) at (0+24,0) {};
        \node[circle, draw = black!100, fill=black!00, inner sep=2pt, minimum width=2pt] (11) at (2+24,0) {};
        \node[circle, draw = black!100, fill=black!00, inner sep=2pt, minimum width=2pt] (12) at (4+24,0) {};

	\draw[-]    (1) -- (2) ;
        \draw[-]    (2) -- (3) ;

        \draw[decorate]    (4) -- (5) ;
        \draw[-]    (5) -- (6) ;

        \draw[-]    (7) -- (8) ;
        \draw[decorate]    (8) -- (9) ;

        \draw[decorate]    (10) -- (11) ;
        \draw[decorate]    (11) -- (12) ;

        
        \end{tikzpicture}
        \medskip
        
        \end{center}

\begin{center}
        \medskip
        
        \begin{tikzpicture}[thick,scale=0.5]
        \tikzset{decoration={snake,amplitude=.4mm,segment length=2mm, post length=0mm,pre length=0mm}}
	
        \node[circle, draw = black!100, fill=black!100, inner sep=2pt, minimum width=2pt] (1) at (0,0) {};
        \node[circle, draw = black!100, fill=black!00, inner sep=2pt, minimum width=2pt] (2) at (2,0) {};
        \node[circle, draw = black!100, fill=black!00, inner sep=2pt, minimum width=2pt] (3) at (4,0) {};

        \node[circle, draw = black!100, fill=black!100, inner sep=2pt, minimum width=2pt] (4) at (0+8,0) {};
        \node[circle, draw = black!100, fill=black!00, inner sep=2pt, minimum width=2pt] (5) at (2+8,0) {};
        \node[circle, draw = black!100, fill=black!00, inner sep=2pt, minimum width=2pt] (6) at (4+8,0) {};

        \node[circle, draw = black!100, fill=black!00, inner sep=2pt, minimum width=2pt] (7) at (0+16,0) {};
        \node[circle, draw = black!100, fill=black!00, inner sep=2pt, minimum width=2pt] (8) at (2+16,0) {};
        \node[circle, draw = black!100, fill=black!100, inner sep=2pt, minimum width=2pt] (9) at (4+16,0) {};

        \node[circle, draw = black!100, fill=black!00, inner sep=2pt, minimum width=2pt] (10) at (0+24,0) {};
        \node[circle, draw = black!100, fill=black!00, inner sep=2pt, minimum width=2pt] (11) at (2+24,0) {};
        \node[circle, draw = black!100, fill=black!100, inner sep=2pt, minimum width=2pt] (12) at (4+24,0) {};

	\draw[-]    (1) -- (2) ;
        \draw[-]    (2) -- (3) ;
        \draw[<-]    (1) edge[bend left] (2) ;

        \draw[-]    (4) -- (5) ;
        \draw[decorate]    (5) -- (6) ;
        \draw[<-]    (4) edge[bend left] (5) ;

        \draw[-]    (7) -- (8) ;
        \draw[-]    (8) -- (9) ;
        \draw[->]    (8) edge[bend left] (9) ;

        \draw[decorate]    (10) -- (11) ;
        \draw[-]    (11) -- (12) ;
        \draw[->]    (11) edge[bend left] (12) ;

        
        \end{tikzpicture}
        \medskip
        
        \end{center}

\begin{center}
        \medskip
        
        \begin{tikzpicture}[thick,scale=0.5]
        \tikzset{decoration={snake,amplitude=.4mm,segment length=2mm, post length=0mm,pre length=0mm}}
	
        \node[circle, draw = black!100, fill=black!00, inner sep=2pt, minimum width=2pt] (1) at (0,0) {};
        \node[circle, draw = black!100, fill=black!100, inner sep=2pt, minimum width=2pt] (2) at (2,0) {};
        \node[circle, draw = black!100, fill=black!00, inner sep=2pt, minimum width=2pt] (3) at (4,0) {};

        \node[circle, draw = black!100, fill=black!00, inner sep=2pt, minimum width=2pt] (4) at (0+8,0) {};
        \node[circle, draw = black!100, fill=black!100, inner sep=2pt, minimum width=2pt] (5) at (2+8,0) {};
        \node[circle, draw = black!100, fill=black!00, inner sep=2pt, minimum width=2pt] (6) at (4+8,0) {};

        \node[circle, draw = black!100, fill=black!00, inner sep=2pt, minimum width=2pt] (7) at (0+16,0) {};
        \node[circle, draw = black!100, fill=black!100, inner sep=2pt, minimum width=2pt] (8) at (2+16,0) {};
        \node[circle, draw = black!100, fill=black!00, inner sep=2pt, minimum width=2pt] (9) at (4+16,0) {};

        \node[circle, draw = black!100, fill=black!00, inner sep=2pt, minimum width=2pt] (10) at (0+24,0) {};
        \node[circle, draw = black!100, fill=black!100, inner sep=2pt, minimum width=2pt] (11) at (2+24,0) {};
        \node[circle, draw = black!100, fill=black!00, inner sep=2pt, minimum width=2pt] (12) at (4+24,0) {};

	\draw[-]    (1) -- (2) ;
        \draw[<-]    (2) -- (3) ;
        \draw[->]    (1) edge[bend left] (2) ;

        \draw[-]    (4) -- (5) ;
        \draw[decorate]    (5) -- (6) ;
        \draw[->]    (4) edge[bend left] (5) ;

        \draw[-]    (7) -- (8) ;
        \draw[-]    (8) -- (9) ;
        \draw[<-]    (8) edge[bend left] (9) ;

        \draw[decorate]    (10) -- (11) ;
        \draw[-]    (11) -- (12) ;
        \draw[<-]    (11) edge[bend left] (12) ;

        
        \end{tikzpicture}
        \medskip
        
        \end{center}

\begin{center}
        \medskip
        
        \begin{tikzpicture}[thick,scale=0.5]
        \tikzset{decoration={snake,amplitude=.4mm,segment length=2mm, post length=0mm,pre length=0mm}}
	
        \node[circle, draw = black!100, fill=black!100, inner sep=2pt, minimum width=2pt] (1) at (0,0) {};
        \node[circle, draw = black!100, fill=black!100, inner sep=2pt, minimum width=2pt] (2) at (2,0) {};
        \node[circle, draw = black!100, fill=black!00, inner sep=2pt, minimum width=2pt] (3) at (4,0) {};

        \node[circle, draw = black!100, fill=black!100, inner sep=2pt, minimum width=2pt] (4) at (0+8,0) {};
        \node[circle, draw = black!100, fill=black!00, inner sep=2pt, minimum width=2pt] (5) at (2+8,0) {};
        \node[circle, draw = black!100, fill=black!100, inner sep=2pt, minimum width=2pt] (6) at (4+8,0) {};

        \node[circle, draw = black!100, fill=black!00, inner sep=2pt, minimum width=2pt] (7) at (0+16,0) {};
        \node[circle, draw = black!100, fill=black!100, inner sep=2pt, minimum width=2pt] (8) at (2+16,0) {};
        \node[circle, draw = black!100, fill=black!100, inner sep=2pt, minimum width=2pt] (9) at (4+16,0) {};

        \node[circle, draw = black!100, fill=black!00, inner sep=2pt, minimum width=2pt] (10) at (0+24,0) {};
        \node[circle, draw = black!100, fill=black!100, inner sep=2pt, minimum width=2pt] (11) at (2+24,0) {};
        \node[circle, draw = black!100, fill=black!100, inner sep=2pt, minimum width=2pt] (12) at (4+24,0) {};

	\draw[-]    (1) -- (2) ;
        \draw[-]    (2) -- (3) ;
        \draw[<-]    (1) edge[bend left] (2) ;
        \draw[<-]    (2) edge[bend left] (3) ;

        \draw[-]    (4) -- (5) ;
        \draw[-]    (5) -- (6) ;
        \draw[<-]    (4) edge[bend left] (5) ;
        \draw[->]    (5) edge[bend left] (6) ;

        \draw[-]    (7) -- (8) ;
        \draw[-]    (8) -- (9) ;
        \draw[->]    (7) edge[bend left] (8) ;
        \draw[->]    (8) edge[bend left] (9) ;

        \draw[-]    (10) -- (11) ;
        \draw[-]    (11) -- (12) ;
        \draw[->]    (10) edge[bend left] (11) ;
        \draw[<-]    (11) edge[bend left] (12) ;

        
        \end{tikzpicture}
        \medskip
        
        \end{center}

        Each of these graphs encodes the collection of vertices of the corresponding facet in the triangulation of $\mathcal{C}_{I_{2}}$, from which we can recover the facet-defining equations of the simplex and thereby compute the canonical form $\Omega_{I_2}$. 
\end{example}

Since the triangulation is unimodular, it follows that the normalized volume of $\mathcal{C}_{I_{n}}$ is given by the sum over all graphs $G_S$ that satisfy the properties listed in Lemma~\ref{lem: path tri facets}. 
Using the decomposition of these properties into subgraphs, we can recover the formula for the normalized volume of $\mathcal{C}_{I_{n}}$ given in \cite{kuhne2022faces}. 

\begin{corollary}
\label{cor: volume path}
The normalized volume of $\mathcal{C}_{I_{n}}$ is $4^n$.
\end{corollary}

\begin{proof}
We first deduce a formula for the normalized volume of $\mathcal{C}_{I_{n}}$ by enumerating the facets of the triangulation using Lemma~\ref{lem: path tri facets}.  Then we show that this formula reduces to $4^n$. 

To enumerate the facets via Lemma~\ref{lem: path tri facets}, we first pick a subset $\{i_1,\ldots, i_k\}$ of $[n+1]$ where we assume $i_1 < \cdots < i_k$.  Let $S$ be a facet with $Z_S=\{z_{i_1},\ldots,z_{i_k}\}$.
There is only one possible induced subgraph on node set $[i_1]$. 
The possible number of induced subgraphs of $G_S$ on node set $\{i_k, i_k+1,\ldots, n+1\}$ is the following
\[
\begin{cases}
1,   &   \mbox{ if } i_k = n+1,\\
2^{n-i_k},   &   \mbox{ if } i_k < n+1.
\end{cases}
\]
Given an interval $\{i_j, i_j+1,\ldots, i_{j+1}\}$, the possible induced subgraphs on this interval of $G_S$ must contain exactly one single edge between $i_j +s$ and $i_j + s + 1$ for $s\in\{0,1,\ldots, i_{j+1} - i_j -1\}$. 
By Lemma~\ref{lem: path tri facets} we always have two choices for this edge. 
Further, by the same lemma, when $s = 0$, there are exactly two possible subgraphs of $G_S$ on this interval. 
For $s > 0$, there are $2^{s}$ choices (including the choice of edge type for the single edge). 
Hence, there are a total of
\begin{equation*}
\begin{split}
2 + \sum_{s = 1}^{i_{j+1}-i_j -1}2^s
&= 1 + \sum_{s = 0}^{i_{j+1}-i_j -1}2^s,\\
&= 1 + (2^{i_{j+1} - i_j}-1),\\
&= 2^{i_{j+1} - i_j}
\end{split}
\end{equation*}
possible subgraphs for this interval.  
The number of facets $S$ of the triangulation with $Z_S = \{i_1,\ldots, i_k\}$ is then equal to 
\[
\begin{cases}
1\cdot \left(\prod_{j=1}^{k-1}2^{i_{j+1} - i_j}\right) \cdot 1,    & \mbox{ if } i_k = n+1,\\
1\cdot \left(\prod_{j=1}^{k-1}2^{i_{j+1} - i_j}\right) \cdot 2^{n-i_k}, & \mbox{ if } i_k < n+1
\end{cases}
= 
\begin{cases}
2^{i_{k} - i_1},   & \mbox{ if } i_k = n+1,\\
2^{n-i_1}, & \mbox{ if } i_k < n+1.
\end{cases}
\]
Summing over all proper subsets of $[n+1]$, yields
\begin{align*}
\sum_{\emptyset\neq Z\in 2^{[n+1]}}2^{n-\min(Z)}+\sum_{Z\in 2^{[n]}} 2^{n+1-\min(Z)}&=\sum_{\ell=1}^n2^{n-\ell}\cdot 2^{n-\ell}+\sum_{\ell=1}^n2^{n-\ell}\cdot 2^{n+1-\ell}+1\\
&=\sum_{\ell=0}^{n-1}4^\ell+2\sum_{\ell=0}^{n-1}4^\ell+1\\
&=3\sum_{\ell=0}^{n-1}4^\ell+1\\
&=3\frac{4^n-1}{4-1}+1=4^n,
\end{align*}
which completes the proof.
\end{proof}






\section{The Cosmological Polytope of the Cycle}
\label{sec: the cycle}
We now consider the cosmological polytope $\mathcal{C}_{C_n}$ associated with the $n$-cycle $C_n$; i.e., the graph with vertex set $V = [n]$ and edge set $E=\{ii+1 ~:~ i\in [n]\}$, where $i+1$ is considered modulo $n$.
Via a mild extension of the observations made in Section~\ref{sec: the path}, we can characterize the facets of a regular unimodular triangulation of $\mathcal{C}_{C_n}$ arising from a good term order. 
This yields a method for computing the canonical form $\Omega_{C_n}$. 
Furthermore, we can enumerate these facets, yielding a closed formula for the normalized volume of $\mathcal{C}_{C_n}$, which was previously unknown. 

We use the notation introduced in \Cref{sec: the path}. 
In particular, we represent sets of variables $S$ in $R_{C_n}$ with the graphs $G_S$.  
For the edge $e=ii+1$ we  also write $y_{ii+1}$ and $y_{i+1i}$ for the corresponding $y$-variables. 
Just as in Section~\ref{sec: the path}, we consider the triangulation $\mathcal{T}$ of $\mathcal{C}_{C_n}$ induced by a lexicographic term order with respect to the following ordering of the variables 
\begin{align}
    \label{eq: order cycle}
    &y_{12}>y_{23}>\cdots y_{n-1n}>y_{n1}>y_{1n}>y_{nn-1}>\cdots>y_{21} > z_{12} > \cdots \\
    &\cdots> z_{n-1n} > t_{12} > \cdots > t_{n-1n} > z_1 > \cdots > z_n \nonumber.
\end{align}
This term order is seen to be a good term order (see \Cref{def : good}).

With respect to such a term order, the leading terms of zig-zag binomials correspond again to partially directed paths ending in a $\circ$, just as in Section~\ref{sec: the path}.  
Note that this is indeed true even more general for facets of the cosmological polytope of any graph for any induced path (whose internal vertices have degree $2$) with respect to any good term order for which variables corresponding to one direction of the path are greater than the ones for the other direction. 

We must now also avoid the subgraphs corresponding to leading terms of cyclic binomials. 
For cycle binomials, this implies that we must avoid subgraphs that are directed cycles (both clockwise and counter-clockwise), such as
\begin{center}
\begin{tikzpicture}[->,thick,scale=.7]
   \node (i) at (90:1cm)  {$1$};
   \node (j) at (-30:1cm) {$2$};
   \node (k) at (210:1cm) {$3$};

   \draw (70:1cm)  arc (70:-10:1cm);
   \draw (-50:1cm) arc (-50:-130:1cm);
   \draw (190:1cm) arc (190:110:1cm);
\end{tikzpicture}
\hspace{1in}
\begin{tikzpicture}[->,thick,scale=.7]
   \node (i1) at (90:1cm)  {$1$};
   \node (j1) at (-30:1cm) {$2$};
   \node (k1) at (210:1cm) {$3$};

   \draw[<-] (70:1cm)  arc (70:-10:1cm);
   \draw[<-] (-50:1cm) arc (-50:-130:1cm);
   \draw[<-] (190:1cm) arc (190:110:1cm);
\end{tikzpicture}.
\end{center}
For cyclic binomials that are not cycle binomials, since $y_{ii+1} > y_{j+1j}$ for every $i,j\in[n]$ (with addition taken modulo $n$), the leading terms will always correspond to partially directed cycles oriented clockwise, such as
\begin{center}
\begin{tikzpicture}[->,thick,scale=.6]
   \node (1) at (90:3cm)  {$1$};
   \node (3) at (30:3cm) {$2$};
   \node (3) at (-30:3cm) {$3$};
   \node (4) at (-90:3cm) {$4$};
   \node (5) at (-150:3cm) {$5$};
   \node (6) at (-210:3cm) {$6$};

   \draw[-] (85:3cm)  arc (85:35:3cm);
   \draw (25:3cm) arc (25:-25:3cm);
   \draw (-35:3cm) arc (-35:-85:3cm);
   \draw[-] (-95:3cm) arc (-95:-145:3cm);
   \draw (-155:3cm) arc (-155:-205:3cm);
   \draw[-] (-215:3cm) arc (-215:-265:3cm);
\end{tikzpicture}.
\end{center}
Hence, we must avoid subgraphs that are partially directed cycles with a clockwise orientation.
The following theorem provides a characterization of the facets of this triangulation in terms of these forbidden subgraphs.

\begin{theorem}
\label{thm: cycle facets}
Let $S$ be a subset of the generators of the ring $R_{C_{n}}$ and let $Z_S = \{z_{i_1},\ldots, z_{i_k}\}$ where $i_1<\cdots < i_k$.  
Then $S$ is a facet of the triangulation of $\mathcal{C}_{C_{n}}$ corresponding to the lexicographic order induced by \eqref{eq: order cycle} if and only if all of the following hold:
\begin{enumerate}
    \item $Z_S\neq \emptyset$,
    \item the induced subgraph of $G_S$ on $\{i_t, i_t+1,\ldots, i_{t+1}\}$ is of the form described in Theorem~\ref{lem: path tri facets}~(2) for all $t\in[k-1]$, and
    \item the induced subgraph of $G_S$ on $\{i_k, i_k+1 \mod n,\ldots, i_{1}\}$ is of the form described in Theorem~\ref{lem: path tri facets}~(2) where the right-most node is $i_1$.
\end{enumerate}
\end{theorem}

\begin{proof}
Suppose that $S$ is a facet of $\mathcal{T}$.  
Then $|S| = 2n$. 
If $Z_S = \emptyset$, then by the forbidden subgraphs arising from the fundamental binomials, we know that every edge in $G_S$ is a double edge consisting of an undirected edge together with a directed edge. 
This, however, would imply that $G_S$ contains a subgraph corresponding to the leading term of a cyclic  binomial, which is a contradiction. 
Hence, $Z_S\neq \emptyset$. 
The fact that conditions (2) and (3) hold follows from the specified variable ordering and the arguments given in the proof of Theorem~\ref{lem: path tri facets}.

Similarly, the converse follows from the arguments given in the proof of Theorem~\ref{lem: path tri facets}, with the additional observation that the specified paths between any two white nodes contains a single edge and these single edges prevent the existence of clockwise partially directed and directed cycles. 
\end{proof}





Similar to the results in Section~\ref{sec: the path}, we can use the characterization in Theorem~\ref{thm: cycle facets} to enumerate the facets of the triangulation and derive a closed formula for the normalized volume of $\mathcal{C}_{C_n}$. 
\begin{theorem}\label{thm : vol cycle}
    The cosmological polytope of the $n$-cycle $C_n$ has normalized volume
    \[
        \textnormal{Vol}(\mathcal{C}_{C_n})=4^n-2^n.
    \]
\end{theorem}

\begin{proof}
Let $S$ be a facet of the triangulation $\mathcal{T}$ described above and $Z_S = \{z_{i_1},\ldots, z_{i_k}\}$ where $i_1<\cdots < i_k$. 
Then the induced subgraph of $G_S$ on $\{i_\ell,\ldots,i_{\ell+1 \mod k}\}$ for $1\leq \ell\leq k$ is of the form described in \Cref{lem: path tri facets} (2). By the proof of \Cref{cor: volume path} there are $2^{i_{\ell+1\mod k}-i_\ell}$ possible subgraphs for this interval which gives $\prod_{\ell=1}^k 2^{i_{\ell+1\mod k}-i_\ell}=2^n$ possible graphs $G_S$ with a prescribed set of white vertices. Varying the latter over all non-empty subsets of the vertices, we get a total of $(2^n-1)\dot 2^n=4^n-2^n$ possible subgraphs $G_S$. It remains to verify that none of these subgraphs contains a clockwise partially oriented cycles or a completely oriented cycle. 
To see this, it suffices to note that if any induced subgraph on $\{i_\ell,\ldots,i_{\ell+1 \mod k}\}$ for  $1\leq \ell\leq k$ is of the first three types in \Cref{lem: path tri facets} (2), then the unique single edge already prevents the existence of such a cycle. 
However, if all considered subgraphs are of the fourth type then none of the variables $y_{ii+1\mod n}$ is present. 
Hence, $G_S$ neither contains a clockwise partially oriented cycle nor a completely oriented cycle. This finishes the proof.
\end{proof}

\section{The Cosmological Polytope of a Tree}
\label{sec: trees}

The description of the facets for a regular unimodular triangulation arising from a good term order for the path in Section~\ref{sec: the path} can be extended to any tree.  
To do so, we first specify a good term order associated to an arbitrary tree $T$ on node set $[n+1]$ that generalizes the term order used in Section~\ref{sec: the path}.  

Fix a leaf node $r$ of $T$ and consider the associated orientation of $T$, denoted $\overrightarrow{T}$, in which all edges are directed away from $r$.
Let $\preceq_r$ denote the partial order on $[n+1]$ given by the distance of a node $i\in[n+1]$ from $r$; that is, $i\preceq_r j$ whenever the length of the unique directed path from $r$ to $i$ in $\overrightarrow{T}$ is less than or equal to the length of the unique directed path from $r$ to $j$ in $\overrightarrow{T}$. 
Fix a linear extension $<_r$ of $\prec_r$ in the following way:

A \emph{floret} in a rooted directed tree consists of a node $i$ and all of its \emph{children}; i.e., the nodes $j$ such that $i\rightarrow j$ is an edge of the tree. 
Iterating over, $k = 1,2,\ldots $, the distance of a node in $\overrightarrow{T}$ from $r$, we consider all nodes at distance $k-1$.  These nodes have been totally ordered as $i_1 < \cdots < i_t$.  
Iterating over $i_\ell$ for $\ell = 1,\ldots, t$, totally order the children of $i_\ell$.  
Then totally order the children of $i_{\ell+1}$ such that all children of $i_{\ell+1}$ are larger than those of $i_\ell$.  For an example consider \Cref{fig: tree}.

We then totally order the edges of $\overrightarrow{T}$ such that for $(i,j),(s,t)\in E(\overrightarrow{T})$ we have $(i,j) \prec (s,t)$ if 
\begin{enumerate}
    \item $i <_r s$, or
    \item $i = s$ and $j <_r t$. 
\end{enumerate}
Using these orderings, we can then define a total order $<$ of the variables $y_{ij}$, $z_{ij}$, $t_{ij}$ and $z_i$ such that
\begin{enumerate}
    \item If $(i,j),(s,t)\in E(\overrightarrow{T})$ and $(i,j)\prec (s,t)$, then
    \begin{itemize}
        \item $y_{ij} > y_{st}$,
        \item $y_{ts} > y_{ji}$, 
        \item $z_{ij} > z_{st}$, and
        \item $t_{ij} > t_{st}$, 
    \end{itemize}
    \item If $(i,j), (s,t)\in E(\overrightarrow{T})$, then 
    \vspace{-12pt}
    \begin{multicols}{3}
    \begin{itemize}
        \item $y_{ij} > y_{ts}$,
        \item $y_{ij} > z_{st}$, 
        \item $y_{ji} > z_{st}$, 
        \item $y_{ij} > t_{st}$, 
        \item $y_{ji} > t_{st}$,
        \item $y_{ij} > z_s$, 
        \item $y_{ji} > z_s$, 
        \item $z_{ij} > t_{st}$, 
        \item $z_{ij} > z_s$,
        \item $t_{ij} > z_s$, and
    \end{itemize}
    \end{multicols}
    \vspace{-12pt}
    \item If $i <_r j$, then $z_i > z_j$. 
\end{enumerate}

This variable ordering is seen to generalize the variable ordering \eqref{eqn: path variable order}, and the associated lexicographic term order on the monomials in $R_T$ is a good term order. 
Hence, by Corollary~\ref{cor: triangulation}, we obtain a regular unimodular triangulation $\mathcal{T}$ of $\mathcal{C}_T$. 

We note the following property of the chosen edge ordering of $T$.
\begin{lemma}
    \label{lem: largest edge}
    Suppose that $i <_r j$, and let $\pi = \{i_1i_2,\ldots, i_{k-1}i_k\}$ be the unique path in $T$ between $i_1 = i$ and $i_k = j$. 
    Then $(i_2,i_1) \prec (i_{k-1},i_k)$.  
\end{lemma}

\begin{proof}
Note first that since $i <_r j$, the distance from $r$ to $j$ in $T$ is at least the distance from $r$ to $i$ in $T$.  
Suppose that these two distances are equal.  
Let $\alpha = \min_{<_r}\{i_1,\ldots, i_k\}$, and let $\pi_1$ and $\pi_2$ be the unique path between $i_1$ and $\alpha$ and $i_k$ and $\alpha$, respectively.  
We index $\pi_1$ as $\pi_1 = \{i_{0,1}i_{1,1}, i_{1,1}i_{2,1},\ldots, i_{t-1,1}i_{t,1}\}$ and $\pi_2$ as $\pi_2 = \{i_{0,2}i_{1,2}, i_{1,2}i_{2,2},\ldots, i_{t-1,2}i_{t,2}\}$ where $i_{0,1} = i_{0,2} = \alpha$, $i_{t,1} = i_1$ and $i_{t,2} = i_k$. 
Observe that $\pi_1$ and $\pi_2$ have the same length since $i$ and $j$ have the same distance from $r$.  

We claim now that $i_{j,1} <_r i_{j,2}$ for all $j = 1,\ldots, t$.  
To see this, suppose for the sake of contradiction that there exists a $j$ for which $i_{j,1} >_r i_{j,2}$.  
Then the floret for $i_{j,2}$ has its children ordered before that of $i_{j,1}$.  
This implies that $i_{j+1,2} >_r i_{j+1,1}$.
Iterating this argument implies that $i = i_{t,1} >_r i_{t,2} = j$, which is a contradiction.  
Hence, $i_{j,1} <_r i_{j,2}$ for all $j = 1,\ldots, t$.  
It follows that
\[
(i_2, i_1) = (i_{t-1,1}, i_{t,1}) \prec (i_{t-1,2}, i_{t,2}) = (i_{k-1}, i_k), 
\]
as desired. 

Now suppose that the distance from $r$ to $j$ in $T$ is strictly larger than the distance from $r$ to $i$ in $T$. 
Then $\pi_1$ contains only nodes of distance at most $t$ from $r$ and $\pi_2$ contains a node at distance $t+1$ from $r$, for minimally chosen $t$.  
By the chosen edge ordering, the edge in $\pi_2$ from the node at distance $t$ to the node at distance $t+1$ is larger than all edges in $\pi_1$.  This edge is also seen to be equal to, or smaller than $(i_{k-1},i_k)$, which completes the proof. 
\end{proof}

For each edge $ij\in E(T)$ the fundamental binomials in $B_T$ imply that if $S$ is a subset of the generators of $R_T$ corresponding to a face of $\mathcal{T}$, then the graph $G_S$ does not contain any of the subgraphs listed in \eqref{eqn: fundamental subgraphs}.

Similarly, the zig-zag binomials in $B_T$ imply that the graph $G_S$ must not contain certain subgraphs along paths if $S$ is a face of $\mathcal{T}$. These subgraphs generalize the partially directed increasing paths from Section~\ref{sec: the path} and are defined as follows:

Let $i_1,i_k$ be vertices of $T$ such that $i_1 <_r i_k$, let $\pi = \{i_1i_2, i_2i_3, \ldots, i_{k-1}i_k\}$ be the unique path in $T$ between $i_1$ and $i_k$, and let $\alpha = \min_{<_r}\{i_1,\ldots, i_k\}$. 
Further let $\pi_1$ and $\pi_2$ denote the subpaths of $\pi$ between $i_1$ and $\alpha$ and $i_k$ and $\alpha$, respectively. 
Given the ordering of the variables above, the leading terms of any zig-zag binomial for a zig-zag pair on $\pi$ have associated graphs being one of the following:
\begin{enumerate}
    \item Partially directed paths toward $i_1$ ending in $\circ$ that include a directed edge on $\pi_1$ pointing toward $i_1$. 
    \item Partially directed paths toward $i_k$ ending in $\circ$ that include an edge directed toward $i_k$ on $\pi_2$, and
    \item Partially directed paths toward $i_1$ ending in $\circ$ with all edges on $\pi_2$ directed and all edges on $\pi_1$ undirected. 
\end{enumerate}

Observe that the paths in (2) include the paths excluded by the leading terms of zig-zag binomials for the path in Section~\ref{sec: the path} by taking $i_1 = \alpha$. 
To see that these three options contain all possible leading terms of zig-zag binomials, consider a zig-zag pair $(E_1,E_2)$ on the path $\pi$, where $E_1$ are the edges directed toward $i_k$ and $E_2$ are the edges directed toward $i_1$. 
If either $E_1$ contains edges on $\pi_2$ or $E_2$ contains edges on $\pi_1$, then under the given variable ordering one of the $y$-variables represented by these edges is the largest.  
Hence, under the given term order the leading term of the associated zig-zag pair is represented by a graph of type (1) or (2).  

On the other hand, if $E_1$ is all the edges on $\pi_1$ and $E_2$ is all the edges on $\pi_2$, then by Lemma~\ref{lem: largest edge}, the leading term is given by the $y$-variables in $E_2$.  
Hence, such a zig-zag pair is represented by the graphs in (3) listed above. 

We call these paths \emph{zig-zag obstructions}.  
Zig-zag obstructions of the form (i) for $i = 1,2,3$ are called \emph{zig-zag obstructions of type i}. 
\begin{example}
\label{ex: zig-zag obstructions}
In Figure~\ref{fig: zig-zag obstructions} we see four graphs.  The first three each contain a zig-zag obstruction of type 1, 2, and 3, respectively, when considered from left-to-right.  The second graph, which highlights in red a zig-zag obstruction of type 2 also contains three additional zig-zag obstructions (of the same type).  These are given by replacing exactly the one of the directed edges with its undirected version, or alternatively considering the subgraph of the red edges where we forget the least of the two directed edges.  The rightmost graph depicts a graph that contains no zig-zag obstructions. 
\begin{figure}
     \centering
     \begin{tikzpicture}[thick,scale=0.4]
        \tikzset{decoration={snake,amplitude=.4mm,segment length=2mm, post length=0mm,pre length=0mm}}
	
       \node[circle, draw = black!100, fill=black!100, inner sep=2pt, minimum width=2pt] (1) at (0,0) {};
       \node[circle, draw = black!100, fill=black!100, inner sep=2pt, minimum width=2pt] (2) at (2,2) {};
       \node[circle, draw = black!100, fill=black!100, inner sep=2pt, minimum width=2pt] (3) at (2,-2) {};
       \node[circle, draw = black!100, fill=black!100, inner sep=2pt, minimum width=2pt] (4) at (4,4) {};
       \node[circle, draw = black!100, fill=black!100, inner sep=2pt, minimum width=2pt] (5) at (4,-4) {};
       \node[circle, draw = red!100, fill=black!00, inner sep=2pt, minimum width=2pt] (6) at (6,6) {};
        \node[circle, draw = black!100, fill=black!00, inner sep=2pt, minimum width=2pt] (7) at (6,-6) {};

	\draw[-]   (1) edge (2) ;
       \draw[->,red]   (1) edge[bend left] (2) ;
       \draw[-]   (1) edge (3) ;
       \draw[<-]   (1) edge[bend left] (3) ;
       \draw[-,red]   (2) edge (4) ;
       \draw[<-]   (2) edge[bend left] (4) ;
       \draw[-]   (3) edge (5) ;
       \draw[<-]   (3) edge[bend left] (5) ;
       \draw[-,red]   (4) edge (6) ;
       \draw[<-]   (4) edge[bend left] (6) ;
       \draw[decorate]   (5) -- (7) ;

       \node at (0-0.5,-0.75) {\footnotesize$\alpha = 1$} ;
       \node at (2,2-0.75) {\footnotesize$2$} ;
       \node at (2,-2-0.75) {\footnotesize$3$} ;
       \node at (4,4-0.75) {\footnotesize$4$} ;
       \node at (4,-4-0.75) {\footnotesize$5$} ;
       \node at (6+0.75,6-0.75) {\footnotesize$i_1 = 6$} ;
       \node at (6+0.75,-6-0.75) {\footnotesize$i_k = 7$} ;

    \end{tikzpicture}\hspace{-20pt}
    \begin{tikzpicture}[thick,scale=0.4]
        \tikzset{decoration={snake,amplitude=.4mm,segment length=2mm, post length=0mm,pre length=0mm}}
	
       \node[circle, draw = black!100, fill=black!100, inner sep=2pt, minimum width=2pt] (1) at (0,0) {};
       \node[circle, draw = black!100, fill=black!100, inner sep=2pt, minimum width=2pt] (2) at (2,2) {};
       \node[circle, draw = black!100, fill=black!100, inner sep=2pt, minimum width=2pt] (3) at (2,-2) {};
       \node[circle, draw = black!100, fill=black!100, inner sep=2pt, minimum width=2pt] (4) at (4,4) {};
       \node[circle, draw = black!100, fill=black!100, inner sep=2pt, minimum width=2pt] (5) at (4,-4) {};
       \node[circle, draw = black!100, fill=black!00, inner sep=2pt, minimum width=2pt] (6) at (6,6) {};
        \node[circle, draw = red!100, fill=black!00, inner sep=2pt, minimum width=2pt] (7) at (6,-6) {};

	\draw[-]   (1) edge (2) ;
       \draw[->]   (1) edge[bend left] (2) ;
       \draw[-]   (1) edge (3) ;
       \draw[->,red]   (1) edge[bend left] (3) ;
       \draw[decorate]   (2) -- (4) ;
       \draw[-]   (3) edge (5) ;
       \draw[->,red]   (3) edge[bend left] (5) ;
       \draw[-]   (4) edge (6) ;
       \draw[<-]   (4) edge[bend left] (6) ;
       \draw[-,red]   (5) edge (7) ;
       \draw[<-]   (5) edge[bend left] (7) ;

       \node at (0-0.5,-0.75) {\footnotesize$\alpha = 1$} ;
       \node at (2,2-0.75) {\footnotesize$2$} ;
       \node at (2,-2-0.75) {\footnotesize$3$} ;
       \node at (4,4-0.75) {\footnotesize$4$} ;
       \node at (4,-4-0.75) {\footnotesize$5$} ;
       \node at (6+0.75,6-0.75) {\footnotesize$i_1 = 6$} ;
       \node at (6+0.75,-6-0.75) {\footnotesize$i_k = 7$} ;
       
    \end{tikzpicture}\hspace{-20pt}
    \begin{tikzpicture}[thick,scale=0.4]
        \tikzset{decoration={snake,amplitude=.4mm,segment length=2mm, post length=0mm,pre length=0mm}}
	
       \node[circle, draw = black!100, fill=black!100, inner sep=2pt, minimum width=2pt] (1) at (0,0) {};
       \node[circle, draw = black!100, fill=black!100, inner sep=2pt, minimum width=2pt] (2) at (2,2) {};
       \node[circle, draw = black!100, fill=black!100, inner sep=2pt, minimum width=2pt] (3) at (2,-2) {};
       \node[circle, draw = black!100, fill=black!100, inner sep=2pt, minimum width=2pt] (4) at (4,4) {};
       \node[circle, draw = black!100, fill=black!100, inner sep=2pt, minimum width=2pt] (5) at (4,-4) {};
       \node[circle, draw = red!100, fill=black!00, inner sep=2pt, minimum width=2pt] (6) at (6,6) {};
       \node[circle, draw = black!100, fill=black!00, inner sep=2pt, minimum width=2pt] (7) at (6,-6) {};

	\draw[-,red]   (1) edge (2) ;
       \draw[<-]   (1) edge[bend left] (2) ;
       \draw[<-,red]   (1) edge (3) ;
       \draw[-,red]   (2) edge (4) ;
       \draw[<-]   (2) edge[bend left] (4) ;
       \draw[-]   (3) edge (5) ;
       \draw[<-,red]   (3) edge[bend left] (5) ;
       \draw[-,red]   (4) edge (6) ;
       \draw[<-]   (4) edge[bend left] (6) ;
       \draw[-]   (5) edge (7) ;
       \draw[<-,red]   (5) edge[bend left] (7) ;

       \node at (0-0.5,-0.75) {\footnotesize$\alpha = 1$} ;
       \node at (2,2-0.75) {\footnotesize$2$} ;
       \node at (2,-2-0.75) {\footnotesize$3$} ;
       \node at (4,4-0.75) {\footnotesize$4$} ;
       \node at (4,-4-0.75) {\footnotesize$5$} ;
       \node at (6+0.75,6-0.75) {\footnotesize$i_1 = 6$} ;
       \node at (6+0.75,-6-0.75) {\footnotesize$i_k = 7$} ;

    \end{tikzpicture}\hspace{-20pt}
    \begin{tikzpicture}[thick,scale=0.4]
        \tikzset{decoration={snake,amplitude=.4mm,segment length=2mm, post length=0mm,pre length=0mm}}
	
       \node[circle, draw = black!100, fill=black!100, inner sep=2pt, minimum width=2pt] (1) at (0,0) {};
       \node[circle, draw = black!100, fill=black!100, inner sep=2pt, minimum width=2pt] (2) at (2,2) {};
       \node[circle, draw = black!100, fill=black!100, inner sep=2pt, minimum width=2pt] (3) at (2,-2) {};
       \node[circle, draw = black!100, fill=black!100, inner sep=2pt, minimum width=2pt] (4) at (4,4) {};
       \node[circle, draw = black!100, fill=black!100, inner sep=2pt, minimum width=2pt] (5) at (4,-4) {};
       \node[circle, draw = black!100, fill=black!00, inner sep=2pt, minimum width=2pt] (6) at (6,6) {};
       \node[circle, draw = black!100, fill=black!00, inner sep=2pt, minimum width=2pt] (7) at (6,-6) {};

	\draw[-]   (1) edge (2) ;
       \draw[<-]   (1) edge[bend left] (2) ;
       \draw[-]   (1) edge (3) ;
       \draw[->]   (1) edge[bend left] (3) ;
       \draw[-]   (2) edge (4) ;
       \draw[<-]   (2) edge[bend left] (4) ;
       \draw[<-]   (3) edge (5) ;
       \draw[-]   (4) edge (6) ;
       \draw[<-]   (4) edge[bend left] (6) ;
       \draw[-]   (5) edge (7) ;
       \draw[<-]   (5) edge[bend left] (7) ;

       \node at (0-0.5,-0.75) {\footnotesize$\alpha = 1$} ;
       \node at (2,2-0.75) {\footnotesize$2$} ;
       \node at (2,-2-0.75) {\footnotesize$3$} ;
       \node at (4,4-0.75) {\footnotesize$4$} ;
       \node at (4,-4-0.75) {\footnotesize$5$} ;
       \node at (6+0.75,6-0.75) {\footnotesize$i_1 = 6$} ;
       \node at (6+0.75,-6-0.75) {\footnotesize$i_k = 7$} ;

    \end{tikzpicture}
        \caption{The first three graphs are, respectively from left-to-right, examples of zig-zag obstructions of type 1, 2 and 3 with the obstruction depicted in red. The rightmost graph is an example of a graph that contains no zig-zag obstructions. The order $<_r$ is the natural order on the vertex set.}
        \label{fig: zig-zag obstructions}
\end{figure}
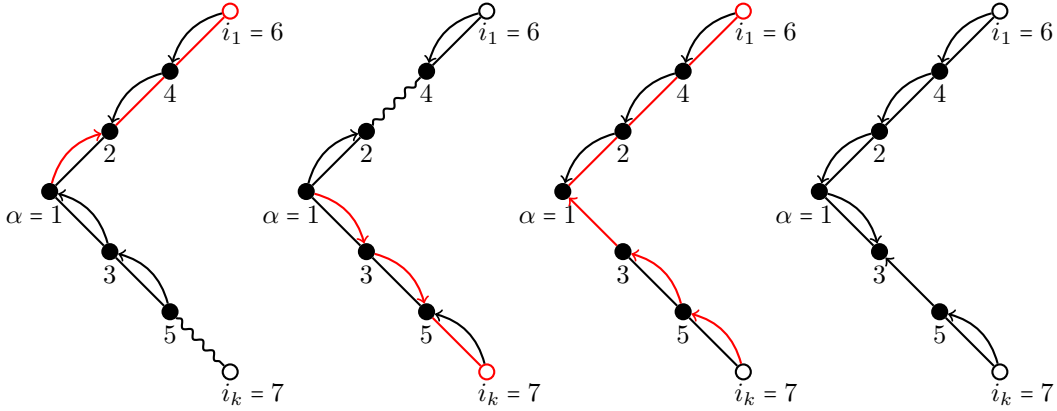
\end{example}

For the chosen term order, we can further reduce the Gr\"obner basis identifed for $I_T$.  
Consider paths of the form $\pi = \{i_1,\ldots, i_k\}$ in which $i_1 <_r i_2 <_r \cdots <_r i_k$.  
We define a \emph{simple zig-zag pair of type 1} as zig-zag pair $(E_1,E_2)$ where
$E_1 = \{i_1 \rightarrow i_{2}\}$ and $E_2 = \{i_{t+1} \rightarrow i_{t}~ :~ t\in\{2,\ldots, k-1\}\}$.
Consider also paths of the form $\pi = \{i_1,i_2,\ldots, i_k\}$ in which $i_j <_r i_1$ for all $j = 2,\ldots, k-1$ but $i_1 <_r i_k$.  
Let $\alpha = \min_{<_r}\{i_1,\ldots, i_k\}$, and take $\pi_1$ and $\pi_2$ as before.  
A \emph{simple zig-zag pair of type 2} is a zig-zag pair $(E_1,E_2)$ on this path where $E_1$ consists of all edges on $\pi_1$ oriented toward $\alpha$ and $E_2$ consists of all edges on $\pi_2$ oriented toward $\alpha$. 


\begin{lemma}
    \label{lem: simple zig-zags}
    Let $T$ be a tree.  The leading term of any zig-zag binomial under the lexicographic order on $R_T$ corresponding to $<$ is divisible by the leading term of a simple zig-zag binomial.
\end{lemma}

\begin{proof}
Under the given term order, the leading term of the zig-zag binomial for a simple zig-zag pair is graphically represented by a partially directed path from $i_1$ to $i_k$ in which the first edge $i_1 \rightarrow i_2$ is directed toward $i_k$ and all other edges are undirected, plus a symbol for the variable $z_{i_1}$.
Given a zig-zag binomial whose leading term is represented by a zig-zag obstruction of type 1, the associated path $\pi = \{i_1,\ldots, i_k\}$ is such that the subpath $\pi_1$ contains at least one directed edge pointing toward $i_1$.  
Pick the edge $i_s\leftarrow i_{s+1}$ of this form on $\pi_1$ with $s$ minimal.  
The remaining edges between this edge and $i_1$ must be undirected, and hence the subpath on $\{i_1,\ldots, i_s\}$ is the graphical representation of the leading term of the zig-zag binomial of a simple zig-zag pair of type 1.
Hence, the leading term of this zig-zag binomial is divisible by the leading term of a zig-zag binomial of a simple zig-zag pair. 
The same argument shows that all zig-zag binomials whose leading terms are represented by zig-zag obstructions of type 2 are also divisible by the leading term of some zig-zag binomial for a simple zig-zag pair of type 1. 

For zig-zag binomials whose leading terms are represented by zig-zag obstructions of type 3, the subpath $\{i_1,\ldots, i_t = \alpha, i_{t+1},\ldots, i_s\}$, where $i_s$ is the first node on $\pi_2$ larger than $i_1$ under $<_r$, is the graphical representation of the leading term of a zig-zag binomial for a simple zig-zag pair of type 2. 
This follows from Lemma~\ref{lem: largest edge}. 
Hence, all such zig-zag binomials also have leading terms divisible by the leading term of a zig-zag binomial for a simple zig-zag pair. 
This completes the proof. 
\end{proof}

Let $S$ be the subset of variables in the leading term of a zig-zag binomial for a simple zig-zag pair of type 1.  
We call the graph $G_S$ a \emph{simple zig-zag obstruction of type 1}. 
We similarly define \emph{simple zig-zag obstructions of type 2}.

\begin{example}
    \label{ex: simple zig-zags}
The zig-zag obstructions in the first and third graphs in Figure~\ref{fig: zig-zag obstructions} are both simple (of type 1 and 2, respectively).  
On the other hand, the zig-zag obstruction in the second graph is not simple, but contains a simple zig-zag obstruction of type 2 as a subgraph.  
This obstruction is given by deleting the first of the two directed edges from the graph.  
Such subgraph inclusions correspond to the divisibility of leading terms as seen in the proof of Lemma~\ref{lem: simple zig-zags}.
\end{example}

Since $T$ is a tree on vertex set $[n+1]$, the dimension of $\mathcal{C}_T$ is $|V| + |E| -1 = 2n$.  
By Lemma~\ref{lem: simple zig-zags}, a characterization of the facets of the triangulation $\mathcal{T}$ of $\mathcal{C}_T$ given by the specified good term order consists of  all subsets $S$ of the variables generating $R_T$ with $|S| = 2n+1$ for which the graph $G_S$ contains no fundamental obstructions and no simple zig-zag obstructions.

In the following, we say that two nodes $i,j$ in a undirected graph $G = (V,E)$ are \emph{connected given a subset} $C\subset V$ if there is a path $\pi = \{i_1i_2,\ldots, i_{k-1}i_k\}$ in $G$ such that $i_1 = i$, $i_k = j$ and $i_2,\ldots, i_{k-1}\notin C$.  
We say a subset of vertices $B$ of $G$ is \emph{maximally connected given} $C$ if all vertices in $B$ are connected given $C$ and there is no pair of vertices $i\in B$ and $j\notin B$ such that $i$ and $j$ are connected given $C$. 
For a tree $T$ and subset $S$ of variables in $R_T$, we let $\overline{G}_{S,1},\ldots, \overline{G}_{S,M}$ denote the induced subgraphs of $G_S$ on the maximally connected subsets of $T$ given $\mathfrak{Z}_S$. 
We call the collection of graphs $\overline{G}_{S,1},\ldots, \overline{G}_{S,M}$ the \emph{$Z_S$-components} of $G_S$.


Recall from Proposition~\ref{prop: connected} that the support graph of $G_S$ for $S$ a facet of the triangulation of $\mathcal{C}_T$ is the tree $T$. 
Hence, the support graph of each $\overline{G}_{S,j}$ is a subtree of $T$. 
In the following, we will want to refer to certain subgraphs of $\overline{G}_{S,j}$ that are induced subgraphs of $\overline{G}_{S,j}$ on the vertex set of the corresponding induced subgraphs of $T$.  
For instance, although a vertex $i$ in $\overline{G}_{S,j}$ may have degree greater than $1$, we will call it a \emph{leaf node} of $\overline{G}_{S,j}$ if it is a leaf node in the support graph of $\overline{G}_{S,j}$.  
Similarly, we may refer to a subgraph of $\overline{G}_{S,j}$ as a \emph{path} in $\overline{G}_{S,j}$ if it is the induced subgraph of $\overline{G}_{S,j}$ on the node set of a path in its support graph, despite the fact that it may include multiple edges between the same pair of vertices.
This mild abuse of terminology should, however, be clear from context.
As a first example, we call the graph $\overline{G}_{S,j}$ \emph{$Z_S$-bounded} if all leaf nodes of $\overline{G}_{S,j}$ are in $\mathfrak{Z}_S$.  
Otherwise, we call it \emph{$Z_S$-unbounded}.

\begin{example}
\label{ex: Z_S-components}
\begin{figure}
     \centering
     \begin{subfigure}[b]{0.45\textwidth}
         \centering
         \begin{tikzpicture}[thick,scale=0.4]
        \tikzset{decoration={snake,amplitude=.4mm,segment length=2mm, post length=0mm,pre length=0mm}}
	
        \node[circle, draw = black!100, fill=black!100, inner sep=2pt, minimum width=2pt] (1) at (0,0) {};
        \node[circle, draw = black!100, fill=black!100, inner sep=2pt, minimum width=2pt] (2) at (2,0) {};
        \node[circle, draw = black!100, fill=black!100, inner sep=2pt, minimum width=2pt] (3) at (4,2) {};
        \node[circle, draw = black!100, fill=black!100, inner sep=2pt, minimum width=2pt] (4) at (4,0) {};
        \node[circle, draw = black!100, fill=black!100, inner sep=2pt, minimum width=2pt] (5) at (6,4) {};
        \node[circle, draw = black!100, fill=black!100, inner sep=2pt, minimum width=2pt] (6) at (6,2) {};
        \node[circle, draw = black!100, fill=black!100, inner sep=2pt, minimum width=2pt] (7) at (6,0) {};
        \node[circle, draw = black!100, fill=black!100, inner sep=2pt, minimum width=2pt] (8) at (6,-2) {};
        \node[circle, draw = black!100, fill=black!100, inner sep=2pt, minimum width=2pt] (9) at (8,6) {};
        \node[circle, draw = black!100, fill=black!100, inner sep=2pt, minimum width=2pt] (10) at (8,4) {};
        \node[circle, draw = black!100, fill=black!100, inner sep=2pt, minimum width=2pt] (11) at (8,2) {};
        \node[circle, draw = black!100, fill=black!100, inner sep=2pt, minimum width=2pt] (12) at (8,0) {};
        \node[circle, draw = black!100, fill=black!100, inner sep=2pt, minimum width=2pt] (13) at (8,-2) {};

        \node[circle, draw = black!100, fill=black!100, inner sep=2pt, minimum width=2pt] (14) at (10,2) {};
        \node[circle, draw = black!100, fill=black!100, inner sep=2pt, minimum width=2pt] (15) at (10,0) {};
        \node[circle, draw = black!100, fill=black!100, inner sep=2pt, minimum width=2pt] (16) at (10,-2) {};

        \node[circle, draw = black!100, fill=black!100, inner sep=2pt, minimum width=2pt] (17) at (12,0) {};
        \node[circle, draw = black!100, fill=black!100, inner sep=2pt, minimum width=2pt] (18) at (12,-2) {};

        \node[circle, draw = black!100, fill=black!100, inner sep=2pt, minimum width=2pt] (19) at (14,2) {};
        \node[circle, draw = black!100, fill=black!100, inner sep=2pt, minimum width=2pt] (20) at (14,-2) {};

        \node[circle, draw = black!100, fill=black!100, inner sep=2pt, minimum width=2pt] (21) at (16,4) {};
        \node[circle, draw = black!100, fill=black!100, inner sep=2pt, minimum width=2pt] (22) at (16,2) {};
        \node[circle, draw = black!100, fill=black!100, inner sep=2pt, minimum width=2pt] (23) at (16,0) {};

        \node[circle, draw = black!100, fill=black!100, inner sep=2pt, minimum width=2pt] (24) at (16,-2) {};
        \node[circle, draw = black!100, fill=black!100, inner sep=2pt, minimum width=2pt] (25) at (16,-4) {};

	\draw[-]   (1) edge (2) ;
        \draw[-]   (2) edge (3) ;
        \draw[-]   (2) edge (4) ;
        \draw[-]   (3) edge (5) ;
        \draw[-]   (4) edge (6) ;
        \draw[-]   (4) edge (7) ;
        \draw[-]   (4) edge (8) ;
        \draw[-]   (5) edge (9) ;
        \draw[-]   (6) edge (10) ;
        \draw[-]   (7) edge (11) ;
        \draw[-]   (7) edge (12) ;
        \draw[-]   (8) edge (13) ;
        \draw[-]   (12) edge (14) ;
        \draw[-]   (12) edge (15) ;
        \draw[-]   (13) edge (16) ;

        \draw[-]   (16) edge (17) ;
        \draw[-]   (16) edge (18) ;

        \draw[-]   (17) edge (19) ;
        \draw[-]   (18) edge (20) ;

        \draw[-]   (19) edge (21) ;
        \draw[-]   (19) edge (22) ;
        \draw[-]   (19) edge (23) ;

        \draw[-]   (20) edge (24) ;
        \draw[-]   (20) edge (25) ;
	 
	\node at (0 + 0.5,0 - 0.5) {\footnotesize $1$} ;
        \node at (2 + 0.5,0 - 0.5) {\footnotesize $2$} ;
        \node at (4 + 0.5,2 - 0.5) {\footnotesize $3$} ;
        \node at (4 + 0.5,0 - 0.5) {\footnotesize $4$} ;
        \node at (6 + 0.5,4 - 0.5) {\footnotesize $5$} ;
        \node at (6 + 0.5,2 - 0.5) {\footnotesize $6$} ;
        \node at (6 + 0.5,0 - 0.5) {\footnotesize $7$} ;
        \node at (6 + 0.5,-2 - 0.5) {\footnotesize $8$} ;
        \node at (8 + 0.5,6 - 0.5) {\footnotesize $9$} ;
        \node at (8 + 0.5,4 - 0.5) {\footnotesize $10$} ;
        \node at (8 + 0.5,2 - 0.5) {\footnotesize $11$} ;
        \node at (8 + 0.5,0 - 0.5) {\footnotesize $12$} ;
        \node at (8 + 0.5,-2 - 0.5) {\footnotesize $13$} ;
        \node at (10 + 0.5,2 - 0.5) {\footnotesize $14$} ;
        \node at (10 + 0.5,0 - 0.5) {\footnotesize $15$} ;
        \node at (10 + 0.5,-2 - 0.5) {\footnotesize $16$} ;
        \node at (12 + 0.5,0 - 0.5) {\footnotesize $17$} ;
        \node at (12 + 0.5,-2 - 0.5) {\footnotesize $18$} ;
        \node at (14 + 0,2 - 0.75) {\footnotesize $19$} ;
        \node at (14 + 0,-2 - 0.75) {\footnotesize $20$} ;
        \node at (16 + 0.5,4 - 0.5) {\footnotesize $21$} ;
        \node at (16 + 0.5,2 - 0.5) {\footnotesize $22$} ;
        \node at (16 + 0.5,0 - 0.5) {\footnotesize $23$} ;
        \node at (16 + 0.5,-2 - 0.5) {\footnotesize $24$} ;
        \node at (16 + 0.5,-4 - 0.5) {\footnotesize $25$} ;
	  	
    \end{tikzpicture}
         \caption{A tree $T$.}
         \label{fig: tree}
     \end{subfigure}
     \hfill
     \begin{subfigure}[b]{0.45\textwidth}
         \centering
         \begin{tikzpicture}[thick,scale=0.4]
        \tikzset{decoration={snake,amplitude=.4mm,segment length=2mm, post length=0mm,pre length=0mm}}
	
        \node[circle, draw = black!100, fill=black!100, inner sep=2pt, minimum width=2pt] (1) at (0,0) {};
        \node[circle, draw = black!100, fill=black!100, inner sep=2pt, minimum width=2pt] (2) at (2,0) {};
        \node[circle, draw = black!100, fill=black!100, inner sep=2pt, minimum width=2pt] (3) at (4,2) {};
        \node[circle, draw = black!100, fill=black!100, inner sep=2pt, minimum width=2pt] (4) at (4,0) {};
        \node[circle, draw = black!100, fill=black!100, inner sep=2pt, minimum width=2pt] (5) at (6,4) {};
        \node[circle, draw = black!100, fill=black!00, inner sep=2pt, minimum width=2pt] (6) at (6,2) {};
        \node[circle, draw = black!100, fill=black!00, inner sep=2pt, minimum width=2pt] (7) at (6,0) {};
        \node[circle, draw = black!100, fill=black!100, inner sep=2pt, minimum width=2pt] (8) at (6,-2) {};
        \node[circle, draw = black!100, fill=black!100, inner sep=2pt, minimum width=2pt] (9) at (8,6) {};
        \node[circle, draw = black!100, fill=black!00, inner sep=2pt, minimum width=2pt] (10) at (8,4) {};
        \node[circle, draw = black!100, fill=black!00, inner sep=2pt, minimum width=2pt] (11) at (8,2) {};
        \node[circle, draw = black!100, fill=black!100, inner sep=2pt, minimum width=2pt] (12) at (8,0) {};
        \node[circle, draw = black!100, fill=black!00, inner sep=2pt, minimum width=2pt] (13) at (8,-2) {};

        \node[circle, draw = black!100, fill=black!00, inner sep=2pt, minimum width=2pt] (14) at (10,2) {};
        \node[circle, draw = black!100, fill=black!100, inner sep=2pt, minimum width=2pt] (15) at (10,0) {};
        \node[circle, draw = black!100, fill=black!100, inner sep=2pt, minimum width=2pt] (16) at (10,-2) {};

        \node[circle, draw = black!100, fill=black!100, inner sep=2pt, minimum width=2pt] (17) at (12,0) {};
        \node[circle, draw = black!100, fill=black!100, inner sep=2pt, minimum width=2pt] (18) at (12,-2) {};

        \node[circle, draw = black!100, fill=black!100, inner sep=2pt, minimum width=2pt] (19) at (14,2) {};
        \node[circle, draw = black!100, fill=black!100, inner sep=2pt, minimum width=2pt] (20) at (14,-2) {};

        \node[circle, draw = black!100, fill=black!00, inner sep=2pt, minimum width=2pt] (21) at (16,4) {};
        \node[circle, draw = black!100, fill=black!00, inner sep=2pt, minimum width=2pt] (22) at (16,2) {};
        \node[circle, draw = black!100, fill=black!00, inner sep=2pt, minimum width=2pt] (23) at (16,0) {};

        \node[circle, draw = black!100, fill=black!00, inner sep=2pt, minimum width=2pt] (24) at (16,-2) {};
        \node[circle, draw = black!100, fill=black!00, inner sep=2pt, minimum width=2pt] (25) at (16,-4) {};

	\draw[-]   (1) edge (2) ;
        \draw[<-]   (1) edge[bend left] (2) ;
        \draw[-]   (2) edge (3) ;
        \draw[<-]   (2) edge[bend left] (3) ;
        \draw[-]   (2) edge (4) ;
        \draw[<-]   (2) edge[bend right] (4) ;
        \draw[-]   (3) edge (5) ;
        \draw[<-]   (3) edge[bend left] (5) ;
        \draw[<-]   (4) edge (6) ;
        \draw[-]   (4) edge (7) ;
        \draw[<-]   (4) edge[bend right] (7) ;
        \draw[decorate]   (4) -- (8) ;
        \draw[-]   (5) edge (9) ;
        \draw[->]   (5) edge[bend left] (9) ;
        \draw[decorate]   (6) -- (10) ;
        \draw[-]   (7) edge (11) ;
        \draw[decorate]   (7) -- (12) ;
        \draw[-]   (8) edge (13) ;
        \draw[<-]   (8) edge[bend right] (13) ;
        \draw[-]   (12) edge (14) ;
        \draw[<-]   (12) edge[bend left] (14) ;
        \draw[-]   (12) edge (15) ;
        \draw[->]   (12) edge[bend right] (15) ;
        \draw[-]   (13) edge (16) ;

        \draw[-]   (16) edge (17) ;
        \draw[<-]   (16) edge[bend left] (17) ;
        \draw[-]   (16) edge (18) ;
        \draw[->]   (16) edge[bend right] (18) ;

        \draw[-]   (17) edge (19) ;
        \draw[<-]   (17) edge[bend right] (19) ;
        \draw[-]   (18) edge (20) ;
        \draw[<-]   (18) edge[bend right] (20) ;

        \draw[decorate]   (19) -- (21) ;
        \draw[-]   (19) edge (22) ;
        \draw[<-]   (19) edge[bend left] (22) ;
        \draw[-]   (19) edge (23) ;

        \draw[<-]   (20) edge (24) ;
        \draw[decorate]   (20) -- (25) ;
	 

    \end{tikzpicture}
         \caption{The graph $G_S$.}
         \label{fig: G_S}
     \end{subfigure}
        \caption{A tree $T$ and the graph $G_S$ for a subset $S$ of the variables in the ring $R_T$.}
        \label{fig: Z_S components ex}
\end{figure}
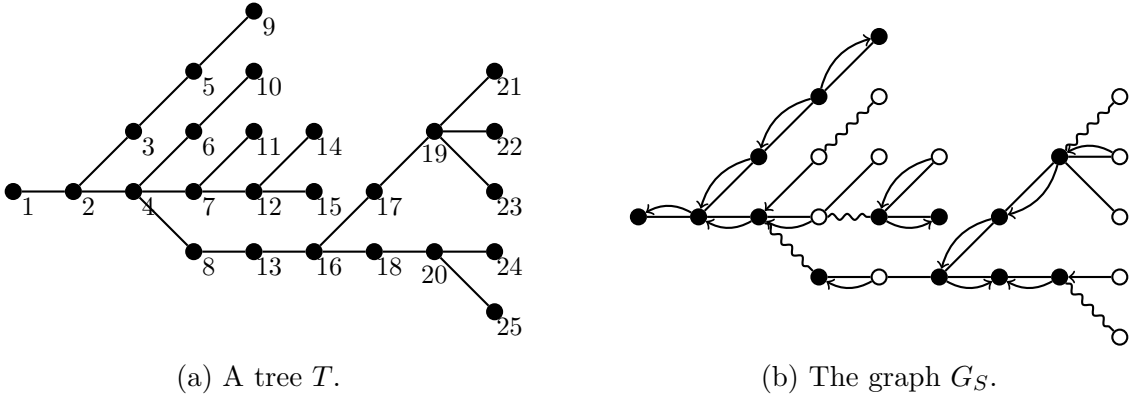
Consider the tree $T$ and the graph $G_S$ depicted in Figure~\ref{fig: tree} and Figure~\ref{fig: G_S}, respectively. For $G_S$, we have that
\[
\mathfrak{Z}_S = \{6,7,10,11,13,14,21,22,23,24,25\}.
\]
The natural order on the vertex set is taken for $<_r$, where $r = 1$, under which we see that $G_S$ contains no fundamental obstructions and no simple zig-zag obstructions.  
From $G_S$ we also see that $|S| = 2n+1$, where $n = 24$, and so it follows that $S$ is a facet of the triangulation $\mathcal{T}$. 

\begin{figure}
     \centering
     \begin{tikzpicture}[thick,scale=0.4]
        \tikzset{decoration={snake,amplitude=.4mm,segment length=2mm, post length=0mm,pre length=0mm}}
	
        \node[circle, draw = black!100, fill=black!100, inner sep=2pt, minimum width=2pt] (1) at (0,0) {};
        \node[circle, draw = black!100, fill=black!100, inner sep=2pt, minimum width=2pt] (2) at (2,0) {};
        \node[circle, draw = black!100, fill=black!100, inner sep=2pt, minimum width=2pt] (3) at (4,2) {};
        \node[circle, draw = black!100, fill=black!100, inner sep=2pt, minimum width=2pt] (4) at (4,0) {};
        \node[circle, draw = black!100, fill=black!100, inner sep=2pt, minimum width=2pt] (5) at (6,4) {};
        \node[circle, draw = black!100, fill=black!00, inner sep=2pt, minimum width=2pt] (6) at (6,2) {};
        \node[circle, draw = black!100, fill=black!00, inner sep=2pt, minimum width=2pt] (7) at (6,0) {};
        \node[circle, draw = black!100, fill=black!100, inner sep=2pt, minimum width=2pt] (8) at (6,-2) {};
        \node[circle, draw = black!100, fill=black!100, inner sep=2pt, minimum width=2pt] (9) at (8,6) {};
        \node[circle, draw = black!100, fill=black!00, inner sep=2pt, minimum width=2pt] (13) at (8,-2) {};

%
%
%
%

	\draw[-]   (1) edge (2) ;
        \draw[<-]   (1) edge[bend left] (2) ;
        \draw[-]   (2) edge (3) ;
        \draw[<-]   (2) edge[bend left] (3) ;
        \draw[-]   (2) edge (4) ;
        \draw[<-]   (2) edge[bend right] (4) ;
        \draw[-]   (3) edge (5) ;
        \draw[<-]   (3) edge[bend left] (5) ;
        \draw[<-]   (4) edge (6) ;
        \draw[-]   (4) -- (7) ;
        \draw[<-]   (4) edge[bend right] (7) ;
        \draw[decorate]   (4) -- (8) ;
        \draw[-]   (5) edge (9) ;
        \draw[->]   (5) edge[bend left] (9) ;
        \draw[-]   (8) edge (13) ;
        \draw[<-]   (8) edge[bend right] (13) ;
%
%
%
%
	 

        \node[circle, draw = black!100, fill=black!00, inner sep=2pt, minimum width=2pt] (7) at (6+4,0) {};
        \node[circle, draw = black!100, fill=black!100, inner sep=2pt, minimum width=2pt] (12) at (8+4,0) {};

        \node[circle, draw = black!100, fill=black!00, inner sep=2pt, minimum width=2pt] (14) at (10+4,2) {};
        \node[circle, draw = black!100, fill=black!100, inner sep=2pt, minimum width=2pt] (15) at (10+4,0) {};
        \draw[decorate]   (7) -- (12) ;
        \draw[-]   (12) edge (14) ;
        \draw[<-]   (12) edge[bend left] (14) ;
        \draw[-]   (12) edge (15) ;
        \draw[->]   (12) edge[bend right] (15) ;
%
%
%
%

    \end{tikzpicture}\hspace{1cm}
    \begin{tikzpicture}[thick,scale=0.4]
        \tikzset{decoration={snake,amplitude=.4mm,segment length=2mm, post length=0mm,pre length=0mm}}
	
        \node[circle, draw = black!100, fill=black!00, inner sep=2pt, minimum width=2pt] (6) at (6,2) {};
        \node[circle, draw = black!100, fill=black!00, inner sep=2pt, minimum width=2pt] (7) at (6,0) {};
        \node[circle, draw = black!100, fill=black!00, inner sep=2pt, minimum width=2pt] (10) at (8,4) {};
        \node[circle, draw = black!100, fill=black!00, inner sep=2pt, minimum width=2pt] (11) at (8,2) {};
        \node[circle, draw = black!100, fill=black!00, inner sep=2pt, minimum width=2pt] (13) at (8,-2) {};

        \node[circle, draw = black!100, fill=black!100, inner sep=2pt, minimum width=2pt] (16) at (10,-2) {};

        \node[circle, draw = black!100, fill=black!100, inner sep=2pt, minimum width=2pt] (17) at (12,0) {};
        \node[circle, draw = black!100, fill=black!100, inner sep=2pt, minimum width=2pt] (18) at (12,-2) {};

        \node[circle, draw = black!100, fill=black!100, inner sep=2pt, minimum width=2pt] (19) at (14,2) {};
        \node[circle, draw = black!100, fill=black!100, inner sep=2pt, minimum width=2pt] (20) at (14,-2) {};

        \node[circle, draw = black!100, fill=black!00, inner sep=2pt, minimum width=2pt] (21) at (16,4) {};
        \node[circle, draw = black!100, fill=black!00, inner sep=2pt, minimum width=2pt] (22) at (16,2) {};
        \node[circle, draw = black!100, fill=black!00, inner sep=2pt, minimum width=2pt] (23) at (16,0) {};

        \node[circle, draw = black!100, fill=black!00, inner sep=2pt, minimum width=2pt] (24) at (16,-2) {};
        \node[circle, draw = black!100, fill=black!00, inner sep=2pt, minimum width=2pt] (25) at (16,-4) {};

        \draw[decorate]   (6) -- (10) ;
        \draw[-]   (7) edge (11) ;
        \draw[-]   (13) edge (16) ;

        \draw[-]   (16) edge (17) ;
        \draw[<-]   (16) edge[bend left] (17) ;
        \draw[-]   (16) edge (18) ;
        \draw[->]   (16) edge[bend right] (18) ;

        \draw[-]   (17) edge (19) ;
        \draw[<-]   (17) edge[bend right] (19) ;
        \draw[-]   (18) edge (20) ;
        \draw[<-]   (18) edge[bend right] (20) ;

        \draw[decorate]   (19) -- (21) ;
        \draw[-]   (19) edge (22) ;
        \draw[<-]   (19) edge[bend left] (22) ;
        \draw[-]   (19) edge (23) ;

        \draw[<-]   (20) edge (24) ;
        \draw[decorate]   (20) -- (25) ;
	 

    \end{tikzpicture}
        \caption{The $Z_S$-components of the graph $G_S$ in Figure~\ref{fig: G_S}.}
        \label{fig: Z_S components of G_S}
\end{figure}
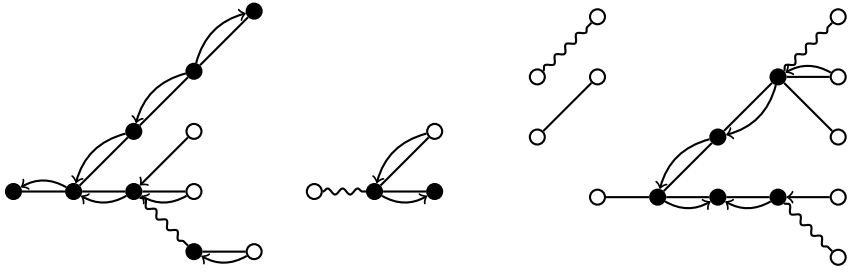
The graph $G_S$ has the $Z_S$-components depicted in Figure~\ref{fig: Z_S components of G_S}, in which the two leftmost graphs are $Z_S$-unbounded and the remaining ones are $Z_S$-bounded.
\end{example}

To analyze the graphs $\overline{G}_{S,j}$ it will be helpful to have a notion of separation of vertices by edges. 
Given a graph $G$ with vertex set $V$ and edge set $E$, we say that the subset of nodes $A\subset V$ are \emph{separated} by the subset of edges $B\subset E$ if for every pair of nodes $i,j\in A$ all paths in $G$ from $i$ to $j$ include an edge in $B$. 
Note that the set $B$ is a cut-set for which the associated cuts each contain a single vertex in $A$. 
The following lemma will be used.
\begin{lemma}
\label{lem: separating tree leaves}
Let $T$ be a tree with set of leaf nodes $A$.  
If the set $A$ is separated by $B$, then $|B| \geq |A|-1$. 
\end{lemma}

\begin{proof}
    Note that if $B$ is a set of edges separating the leaf nodes $A$ of the tree $T$ then deleting the edges in $B$ from $T$ results in a forest with at least $|A|$ connected components such that no two leaves of $T$ are in the same component.  This is because removing a single edge from a forest always increases the number of connected components by exactly 1. Hence, to get $|A|$ connected components that each contain a unique leaf node we must remove $|A| - 1$ edges, which proves the desired lower bound.
\end{proof}

\begin{lemma}
    \label{lem: unique pair}
    Let $T$ be a tree with set of leaves $A$ and let $B$ be a set of edges in $T$ that separate $A$.  If $|B| = |A|-1$, then for each edge $e\in B$ there exists a unique pair of vertices $i,j\in A$ such that $B$ separates $i$ and $j$ but $B\setminus e$ does not. 
\end{lemma}

\begin{proof}
By Lemma~\ref{lem: separating tree leaves} the set $B$ is a minimal separating set for $A$. 
Hence, removing a single edge $e = st\in B$ from $B$ connects at least one pair of leaves of $T$. 
Suppose we connect two pairs of leaves, say $i,j$ and $k, \ell$. 
Then, without loss of generality, $i$ and $k$ and $j$ and $\ell$ are, respectively, on the same side of the single edge $e$; i.e., they are in the same connected component given by deleting the edge $e$. 
Say $j$ and $\ell$ are in the component containing $t$ and $i$ and $k$ are in the component containing $s$. 

Since $T$ is a tree there is a unique path between $i$ and $k$, and this path must be the concatenation of the paths between $i$ and $s$ and $k$ and $s$.  
Since this path does not contain $e$, it must contain another edge $e^\prime\in B$. 
Without loss of generality, suppose $e^\prime$ lies on the path between $i$ and $s$. 
This contradicts the assumption that removing $e$ from the separating set connects $i$ and $j$, which completes the proof. 
\end{proof}

We say that an edge $e$ \emph{critically separates} a pair of nodes $i$ and $j$ in a graph $G = (V,E)$ with respect to $B\subset E$  if $B$ separates $i$ and $j$ but $B\setminus e$ does not. 
Lemma~\ref{lem: unique pair} states that if $G$ is a tree with set of leaf nodes $A$ and $B$ separates $A$ with $|B| = |A| - 1$, then each edge in $B$ critically separates a unique pair of leaf nodes in $G$. 
For a fixed tree $T$, we let $m_j \coloneqq |V(\overline{G}_{S,j})\cap \mathfrak{Z}_S|$.  
The following generalizes Lemma~\ref{lem: facet means connected} to arbitrary trees.
\begin{lemma}
\label{lem: facets means connected for trees}
Let $T$ be a tree on node set $[n+1]$ and let $S$ be a facet of the triangulation of $\mathcal{C}_T$.   
If $\mathfrak{Z}_S = \{i_1 <_r \cdots <_r i_{n+1-k}\}$ 
it follows that
\begin{enumerate}
    \item $G_S$ contains exactly $k$ double edges,
    \item All double edges are of the form
    \begin{center}
    \begin{tikzpicture}[thick,scale=0.5]
        \tikzset{decoration={snake,amplitude=.4mm,segment length=2mm, post length=0mm,pre length=0mm}}
	
        \node[circle, draw = black!00, fill=black!00, inner sep=2pt, minimum width=2pt] (1) at (0,0) {};
        \node[circle, draw = black!00, fill=black!00, inner sep=2pt, minimum width=2pt] (2) at (2,0) {};

        \node[circle, draw = black!00, fill=black!00, inner sep=2pt, minimum width=2pt] (3) at (10,0) {};
        \node[circle, draw = black!00, fill=black!00, inner sep=2pt, minimum width=2pt] (4) at (12,0) {};

	\draw[->]   (1) edge[bend left] (2) ;
        \draw[-]    (1) -- (2) ;

        \draw[<-]   (3) edge[bend left] (4) ;
        \draw[-]    (3) -- (4) ;

        \node at (6,0) {or} ;
    \end{tikzpicture}
    \end{center}
    \item $\overline{G}_{S,j}$ contains exactly $m_j-1$ single edges for all $j\in[M]$. 
\end{enumerate}
\end{lemma}

\begin{proof} 
The fact that (2) holds follows from the first four fundamental obstructions listed in \eqref{eqn: fundamental subgraphs}. 
Hence, all edges in $G_S$ are either single edges or double edges of the form 
\begin{center}
    \begin{tikzpicture}[thick,scale=0.5]
        \tikzset{decoration={snake,amplitude=.4mm,segment length=2mm, post length=0mm,pre length=0mm}}
	
        \node[circle, draw = black!00, fill=black!00, inner sep=2pt, minimum width=2pt] (1) at (0,0) {};
        \node[circle, draw = black!00, fill=black!00, inner sep=2pt, minimum width=2pt] (2) at (2,0) {};

        \node[circle, draw = black!00, fill=black!00, inner sep=2pt, minimum width=2pt] (3) at (10,0) {};
        \node[circle, draw = black!00, fill=black!00, inner sep=2pt, minimum width=2pt] (4) at (12,0) {};

	\draw[->]   (1) edge[bend left] (2) ;
        \draw[-]    (1) -- (2) ;

        \draw[<-]   (3) edge[bend left] (4) ;
        \draw[-]    (3) -- (4) ;

        \node at (6,0) {or} ;
    \end{tikzpicture}
    \end{center}

We now show that (1) holds.  
Since $S$ is a facet and $\dim(\mathcal{C}_T) = |V(T)| + |E(T)| - 1 = 2n$, we have  $|S| = 2n+1$. 
Suppose now that $|\mathfrak{Z}_S| = n-k+1$.  
We know that the support graph of  $G_S$ equals $G$ by Proposition~\ref{prop: connected}, and hence $G_S$ contains at least $n$ edges.  
Since $|\mathfrak{Z}_S| + n = 2n-k +1$ and $|S| = 2n+1$, and every edge in $G_S$ is either a single edge or a double edge, it follows that $G_S$ contains exactly $k$ double edges. 

It remains to see that (3) holds. 
Consider first the case when $\overline{G}_{S,j}$ is $Z_S$-bounded.  
Let $n_j = |V(\overline{G}_{S,j})|$. 
We claim that there are at most $n_j - m_j$ double edges in $\overline{G}_{S,j}$; equivalently,  there are at least $m_j - 1$ single edges in $\overline{G}_{S,j}$.  
To see this, note that the single edges in $\overline{G}_{S,j}$ must separate all nodes in $V(\overline{G}_{S,j})\cap\mathfrak{Z}_S$.  
That is, the edges in the support graph of $\overline{G}_{S,j}$ corresponding to the single edges in $\overline{G}_{S,j}$ must separate the leaf nodes of the support graph. 
Otherwise, there is at least one simple zig-zag obstruction in $G_S$, which would contradict $S$ being a facet. 

To see this claim, assume otherwise.  
Then there is a pair of vertices $i_1,i_k\in V(\overline{G}_{S,j})\cap\mathfrak{Z}_S$ such that the unique path $\pi = \{i_1i_2,\ldots, i_{k-1}i_k\}$ in $T$ between $i_1$ and $i_k$ contains only double edges in $\overline{G}_{S,j}$.  
Suppose without loss of generality that $i_1 <_r i_k$.  
Since $i_1,i_k\in\mathfrak{Z}_S$, we know these nodes are $\circ$ nodes in $G_S$.  
Let $\alpha = \min_{<_r}\{i_1,\ldots, i_k\}$, and let $\pi_1$ and $\pi_2$ denote the subpaths of $\pi$ between $i_1$ and $\alpha$ and $i_k$ and $\alpha$, respectively. 
If there is a directed arrow pointing toward $i_1$ or $i_k$, respectively on $\pi_1$ or $\pi_2$ then $G_S$ contains a zig-zag obstruction of type $(1)$ or $(2)$, and hence contains a simple zig-zag obstruction of type 1.  
This, would contradict $S$ being a facet. 
Thus, all directed edges on $\pi$ must be directed toward $\alpha$.
However, this implies that $G_S$ contains a simple zig-zag obstruction of type 2, which is again a contradiction. 

Thus, no such paths of double edges exist, and we conclude that the single edges must separate the nodes in $V(\overline{G}_{S,j})\cap\mathfrak{Z}_S$.  
Since $T$ is a tree, we require at least $m_j - 1$ such single edges in $\overline{G}_{S,j}$ by Lemma~\ref{lem: separating tree leaves}.  
Similarly, if $\overline{G}_{S,j}$ is $Z_S$-unbounded, we can consider the induced subgraph of $\overline{G}_{S,j}$ by all nodes on the unique paths in $T$ between the vertices in $V(\overline{G}_{S,j})\cap\mathfrak{Z}_S$, and the same argument applies. 
Hence, there are at least $m_j-1$ single edges in $\overline{G}_{S,j}$ and at most $n_j - m_j$ double edges.

We now claim that there are exactly $m_j - 1$ single edges in $\overline{G}_{S,j}$. 
By our choice of Gr\"obner basis, the graph $\overline{G}_{S,j}$ also corresponds to a facet of the triangulation of the cosmological polytope of the support graph $T^\prime$ of $\overline{G}_{S,j}$.  
Since $T^\prime$ is a tree on $n_j$ vertices, and since $|V(\overline{G}_{S,j})\cap \mathfrak{Z}_S|$ contains $m_j$ vertices, we know from (1) that $\overline{G}_{S,j}$ must exactly contain $n_j - m_j$ double edges.
Equivalently, it must contain exactly $m_j-1$ single edges, which completes the proof. 
\end{proof}

In the proof of Lemma~\ref{lem: facets means connected for trees} we use the fact that the set of single edges in a $Z_S$-component $\overline{G}_{S,j}$ corresponds to a set of edges in the support graph of the component that separates its leaf nodes. 
In the following, we will simply say that a set of edges in a $Z_S$-component \emph{separates} a set of nodes $A$ if the corresponding edges in the support graph separate $A$. 

By definition, each of the $Z_S$-components $\overline{G}_{S,j}$ has a unique minimal node $r_{S,j}$ under the vertex ordering $<_r$.  
This node is the root of the induced subgraph of $\overrightarrow{T}$ on the vertex set $V(\overline{G}_{S,j})$.  
A root-to-leaf path in this subtree is a path connecting the root node $r_{S,j}$ to a leaf node $i$.  
We can consider the induced subgraph on the node set of such a path in the graph $\overline{G}_{S,j}$.  
For simplicity, we refer to such a subpath as a root-to-leaf path in $\overline{G}_{S,j}$, noting that it may contain multiple edges.  
When we refer to a leaf-to-root path we imagine reading such a root-to-leaf path backwards from the leaf to the root node. 

For each node $i\in V(\overline{G}_{S,j})\cap(\mathfrak{Z}_S\setminus\{r_{S,j}\})$ consider the first single edge encountered along the leaf-to-root path from $i$ to $r_{S,j}$ in $\overline{G}_{S,j}$.  
By Lemma~\ref{lem: unique pair}, this edge critically separates a unique pair of vertices $s,t\in V(\overline{G}_{S,j})\cap\mathfrak{Z}_S$ (with respect to the set of single edges in $\overline{G}_{S,j}$). 
Moreover, it can be seen by applying a similar argument as in the proof of Lemma~\ref{lem: unique pair} that one of these two vertices must be $i$, say $s = i$. 
Let $\pi = \{i_1i_2,\ldots, i_{k-1}i_k\}$ denote the unique path in $T$ between $i$ and $t$, and let $\alpha = \min_{<_r}\{i_1,\ldots, i_k\}$. 
We call the path between $\alpha$ and $i$ the \emph{threshold path for $i$}.
If we take $i = i_k$ and $\alpha = i_t$ for some $t\in[k]$ in $\pi$, we see that $i_{t} <_r i_{t+1} <_r \cdots <_r i_k$; that is, the threshold path is a decreasing path from $i$ to $\alpha$ under the ordering $<_r$. 
It follows that the threshold path always contains the unique single edge on $\pi$ separating $i$ and $t$ as the leaf-to-root path used to define the threshold path is also decreasing. 

If $i<_r t$ it follows that the threshold path is the path $\pi_1$, and if $t <_r i$ it is $\pi_2$. 
In the former case, we say the threshold path is \emph{type 1}, and we say it is \emph{type 2} in the latter case. 
A threshold path of type 1 is \emph{blocking} if the portion of the path between $i$ and the first single edge consists only of undirected edges paired with directed edges pointing away from $i$ and the single edge is a directed edge pointing away from $i$ or a $\middlewave{0.5cm}$. 
A threshold path of type 2 is \emph{blocking} if the portion of the path between $i$ and the first single edge consists only of undirected edges paired with directed edges pointing away from $i$ and one of the following holds:
\begin{enumerate}
    \item The single edge is undirected and all directed edges on the threshold path point away from $i$, or
    \item the single edge is $\middlewave{0.5cm}$, or
    \item the single edge is a directed edge pointing away from $i$ and at least one directed edge on the portion of the path between $\alpha$ and this single edge points toward $i$. 
\end{enumerate}

\begin{example}
    \label{ex: threshold paths}
    We consider some examples of threshold paths in the rightmost $Z_S$-component for the graph $G_S$ in Example~\ref{ex: Z_S-components}. 
    For instance, the threshold path for the vertex 22 in this $Z_S$ component is 
    \begin{center}
    \begin{tikzpicture}[thick,scale=0.4]
        \tikzset{decoration={snake,amplitude=.4mm,segment length=2mm, post length=0mm,pre length=0mm}}
	
        \node[circle, draw = black!100, fill=black!00, inner sep=2pt, minimum width=2pt] (13) at (8,-2) {};

        \node[circle, draw = black!100, fill=black!100, inner sep=2pt, minimum width=2pt] (16) at (10,-2) {};

        \node[circle, draw = black!100, fill=black!100, inner sep=2pt, minimum width=2pt] (17) at (12,0) {};

        \node[circle, draw = black!100, fill=black!100, inner sep=2pt, minimum width=2pt] (19) at (14,2) {};

        \node[circle, draw = black!100, fill=black!00, inner sep=2pt, minimum width=2pt] (22) at (16,2) {};


        \draw[-]   (13) edge (16) ;

        \draw[-]   (16) edge (17) ;
        \draw[<-]   (16) edge[bend left] (17) ;

        \draw[-]   (17) edge (19) ;
        \draw[<-]   (17) edge[bend right] (19) ;

        \draw[-]   (19) edge (22) ;
        \draw[<-]   (19) edge[bend left] (22) ;

	 

    \end{tikzpicture}
    \end{center}
    The unique edge on the leaf-to-root path from $22$ to $r = 1$ has first single edge the undirected edge between vertices $13$ and $16$ depicted here.  
    This edge critically separates the two nodes $13, 22\in\mathfrak{Z}_S$ (with respect to the set of all single edges in the $Z_S$-component), and hence the threshold path for $22$ is the entire path in the $Z_S$-component between $22$ to $13$. 
    Since $13 <_r 22$, this is a threshold path of type 2. 
    We see that it is blocking since it satisfies (1).  

    As a second example, consider the leaf node $25$ in the same $Z_S$-component.  The first edge on its leaf-to-root path is a single edge.  This edge critically separates nodes $22, 25\in\mathfrak{Z}_S$.  The path between these nodes is depicted on the left in the following, and the threshold path for $25$ is depicted on the right:
    \begin{center}
        \begin{tikzpicture}[thick,scale=0.4]
        \tikzset{decoration={snake,amplitude=.4mm,segment length=2mm, post length=0mm,pre length=0mm}}
	

        \node[circle, draw = black!100, fill=black!100, inner sep=2pt, minimum width=2pt] (16) at (10,-2) {};

        \node[circle, draw = black!100, fill=black!100, inner sep=2pt, minimum width=2pt] (17) at (12,0) {};
        \node[circle, draw = black!100, fill=black!100, inner sep=2pt, minimum width=2pt] (18) at (12,-2) {};

        \node[circle, draw = black!100, fill=black!100, inner sep=2pt, minimum width=2pt] (19) at (14,2) {};
        \node[circle, draw = black!100, fill=black!100, inner sep=2pt, minimum width=2pt] (20) at (14,-2) {};

        \node[circle, draw = black!100, fill=black!00, inner sep=2pt, minimum width=2pt] (22) at (16,2) {};

        \node[circle, draw = black!100, fill=black!00, inner sep=2pt, minimum width=2pt] (25) at (16,-4) {};


        \draw[-]   (16) edge (17) ;
        \draw[<-]   (16) edge[bend left] (17) ;
        \draw[-]   (16) edge (18) ;
        \draw[->]   (16) edge[bend right] (18) ;

        \draw[-]   (17) edge (19) ;
        \draw[<-]   (17) edge[bend right] (19) ;
        \draw[-]   (18) edge (20) ;
        \draw[<-]   (18) edge[bend right] (20) ;

        \draw[-]   (19) edge (22) ;
        \draw[<-]   (19) edge[bend left] (22) ;

        \draw[decorate]   (20) -- (25) ;
	 

    \end{tikzpicture}\hspace{1cm}
    \begin{tikzpicture}[thick,scale=0.4]
        \tikzset{decoration={snake,amplitude=.4mm,segment length=2mm, post length=0mm,pre length=0mm}}
	

        \node[circle, draw = black!100, fill=black!100, inner sep=2pt, minimum width=2pt] (16) at (10,-2) {};

        \node[circle, draw = black!100, fill=black!100, inner sep=2pt, minimum width=2pt] (18) at (12,-2) {};

        \node[circle, draw = black!100, fill=black!100, inner sep=2pt, minimum width=2pt] (20) at (14,-2) {};


        \node[circle, draw = black!100, fill=black!00, inner sep=2pt, minimum width=2pt] (25) at (16,-4) {};


        \draw[-]   (16) edge (18) ;
        \draw[->]   (16) edge[bend right] (18) ;

        \draw[-]   (18) edge (20) ;
        \draw[<-]   (18) edge[bend right] (20) ;


        \draw[decorate]   (20) -- (25) ;
	 

    \end{tikzpicture}
    \end{center}
    Since $22 <_r 25$, the threshold path for $25$ is also type 2, and we see that it is blocking by (2). 

    As a third, and final example, consider the root-to-leaf path from node $21$ in the same $Z_S$-component. 
    As the first edge on this path is a single edge which critically separates nodes $21$ and $22$, we have that the path between nodes $21$ and $22$ is that depicted on the left, and the threshold path for $21$ is that depicted on the right:
    \begin{center}
    \begin{tikzpicture}[thick,scale=0.4]
        \tikzset{decoration={snake,amplitude=.4mm,segment length=2mm, post length=0mm,pre length=0mm}}
	
        \node[circle, draw = black!100, fill=black!100, inner sep=2pt, minimum width=2pt] (19) at (14,2) {};
        \node[circle, draw = black!100, fill=black!00, inner sep=2pt, minimum width=2pt] (21) at (16,4) {};
        \node[circle, draw = black!100, fill=black!00, inner sep=2pt, minimum width=2pt] (22) at (16,2) {};

        \draw[decorate]   (19) -- (21) ;
        \draw[-]   (19) edge (22) ;
        \draw[<-]   (19) edge[bend left] (22) ;
	 

    \end{tikzpicture}\hspace{1cm}
    \begin{tikzpicture}[thick,scale=0.4]
        \tikzset{decoration={snake,amplitude=.4mm,segment length=2mm, post length=0mm,pre length=0mm}}
	
        \node[circle, draw = black!100, fill=black!100, inner sep=2pt, minimum width=2pt] (19) at (14,2) {};
        \node[circle, draw = black!100, fill=black!00, inner sep=2pt, minimum width=2pt] (21) at (16,4) {};

        \draw[decorate]   (19) -- (21) ;
	 

    \end{tikzpicture}
    \end{center}
    Since $21 <_r 22$ this is a threshold path of type 1, which is seen to be blocking since the single edge is a $\middlewave{0.5cm}$. 
\end{example}

We will use the notion of blocking paths in our characterization of the facets of the triangulation of $\mathcal{C}_T$. 
We additionally require one more type of path. 
Given a node $i$ of the graph $T$ we say that a node $j$ \emph{covers} $i$ if $i <_r j$ and no node along the unique path from $j$ to $r$ in $T$ is larger than $i$. 
Let $\pi = \{i_1i_2,\ldots, i_{k-1}i_k\}$ be the unique path in $T$ between $i$ and $j$ and let $\alpha = \min_{<_r}\{i_1,\ldots, i_k\}$. 
A \emph{partially directed branching} from $j$ to $i$ is a partial orientation of $\pi$ such that all edges along the path between $i$ and $\alpha$ are undirected and all edges along the path between $j$ and $\alpha$ are directed toward $\alpha$. 

\begin{example}
    \label{ex: partially directed branching}

    \begin{figure}
     \centering
     \begin{subfigure}[b]{0.45\textwidth}
         \centering
         \begin{tikzpicture}[thick,scale=0.4]
        \tikzset{decoration={snake,amplitude=.4mm,segment length=2mm, post length=0mm,pre length=0mm}}
	
        \node[circle, draw = black!100, fill=black!00, inner sep=2pt, minimum width=2pt] (13) at (8,-2) {};

        \node[circle, draw = black!100, fill=black!100, inner sep=2pt, minimum width=2pt] (16) at (10,-2) {};

        \node[circle, draw = black!100, fill=black!100, inner sep=2pt, minimum width=2pt] (17) at (12,0) {};
        \node[circle, draw = black!100, fill=black!100, inner sep=2pt, minimum width=2pt] (18) at (12,-2) {};

        \node[circle, draw = black!100, fill=black!100, inner sep=2pt, minimum width=2pt] (19) at (14,2) {};
        \node[circle, draw = black!100, fill=black!100, inner sep=2pt, minimum width=2pt] (20) at (14,-2) {};

        \node[circle, draw = black!100, fill=black!00, inner sep=2pt, minimum width=2pt] (21) at (16,4) {};
        \node[circle, draw = black!100, fill=black!00, inner sep=2pt, minimum width=2pt] (22) at (16,2) {};
        \node[circle, draw = black!100, fill=black!00, inner sep=2pt, minimum width=2pt] (23) at (16,0) {};

        \node[circle, draw = black!100, fill=black!00, inner sep=2pt, minimum width=2pt] (24) at (16,-2) {};
        \node[circle, draw = black!100, fill=black!00, inner sep=2pt, minimum width=2pt] (25) at (16,-4) {};

        \draw[-]   (13) edge (16) ;

        \draw[-]   (16) edge (17) ;
        \draw[<-]   (16) edge[bend left] (17) ;
        \draw[-]   (16) edge (18) ;
        \draw[->]   (16) edge[bend right] (18) ;

        \draw[-]   (17) edge (19) ;
        \draw[<-]   (17) edge[bend right] (19) ;
        \draw[-]   (18) edge (20) ;
        \draw[<-]   (18) edge[bend right] (20) ;

        \draw[decorate]   (19) -- (21) ;
        \draw[-]   (19) edge (22) ;
        \draw[<-]   (19) edge[bend left] (22) ;
        \draw[-]   (19) edge (23) ;

        \draw[<-]   (20) edge (24) ;
        \draw[decorate]   (20) -- (25) ;
	 

    \end{tikzpicture}
         \caption{A $Z_S$-component of the graph $G_S$ in Figure~\ref{fig: G_S}.}
         \label{fig: Z_S}
     \end{subfigure}
     \hfill
     \begin{subfigure}[b]{0.45\textwidth}
         \centering
         \begin{tikzpicture}[thick,scale=0.4]
        \tikzset{decoration={snake,amplitude=.4mm,segment length=2mm, post length=0mm,pre length=0mm}}
	
        \node[circle, draw = black!100, fill=black!00, inner sep=2pt, minimum width=2pt] (13) at (8,-2) {};

        \node[circle, draw = black!100, fill=black!100, inner sep=2pt, minimum width=2pt] (16) at (10,-2) {};

        \node[circle, draw = black!100, fill=black!100, inner sep=2pt, minimum width=2pt] (17) at (12,0) {};
        \node[circle, draw = black!100, fill=black!100, inner sep=2pt, minimum width=2pt] (18) at (12,-2) {};

        \node[circle, draw = black!100, fill=black!100, inner sep=2pt, minimum width=2pt] (19) at (14,2) {};
        \node[circle, draw = black!100, fill=black!100, inner sep=2pt, minimum width=2pt] (20) at (14,-2) {};

        \node[circle, draw = black!100, fill=black!00, inner sep=2pt, minimum width=2pt] (21) at (16,4) {};
        \node[circle, draw = black!100, fill=black!00, inner sep=2pt, minimum width=2pt] (22) at (16,2) {};
        \node[circle, draw = black!100, fill=black!00, inner sep=2pt, minimum width=2pt] (23) at (16,0) {};

        \node[circle, draw = black!100, fill=black!00, inner sep=2pt, minimum width=2pt] (24) at (16,-2) {};
        \node[circle, draw = black!100, fill=black!00, inner sep=2pt, minimum width=2pt] (25) at (16,-4) {};

        \draw[-]   (13) edge (16) ;

        \draw[-,red]   (16) edge (17) ;
        \draw[<-]   (16) edge[bend left] (17) ;
        \draw[-]   (16) edge (18) ;
        \draw[<-,red]   (16) edge[bend right] (18) ;

        \draw[-,red]   (17) edge (19) ;
        \draw[<-]   (17) edge[bend right] (19) ;
        \draw[-]   (18) edge (20) ;
        \draw[<-,red]   (18) edge[bend right] (20) ;

        \draw[decorate]   (19) -- (21) ;
        \draw[-]   (19) edge (22) ;
        \draw[<-]   (19) edge[bend left] (22) ;
        \draw[-,red]   (19) edge (23) ;

        \draw[<-,red]   (20) edge (24) ;
        \draw[decorate]   (20) -- (25) ;
	 

    \end{tikzpicture}
         \caption{The $Z_S$-component in Figure~\ref{fig: Z_S} with one edge direction reversed.}
         \label{fig: Z_S aug}
     \end{subfigure}
        \caption{Examples and non-examples of partially directed branchings.  An example of a partially directed branching appears in red.}
        \label{fig: covers and branching}
\end{figure}
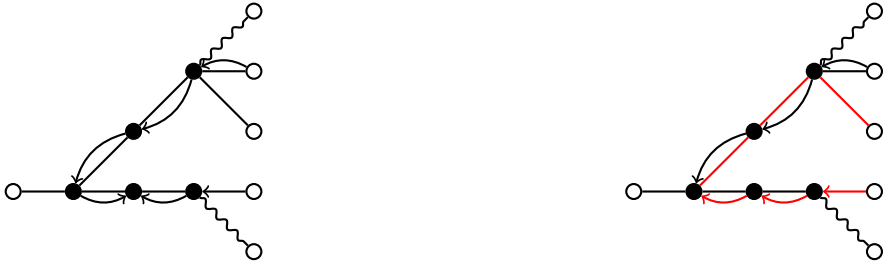

    We consider the $Z_S$-component depicted in Figure~\ref{fig: Z_S} for the graph $G_S$ in Figure~\ref{fig: G_S}.
    The node $23$ is covered by nodes $24$ and $25$. We see from inspection of the graph that there are no partially directed branchings in this $Z_S$-component from $24$ to $23$ nor from $25$ to $23$.  
    However, if we were to reverse the direction of the edge between nodes $16$ and $18$ we would then have a partially directed branching from $24$ to $23$, as depicted in red in Figure~\ref{fig: Z_S aug}.
\end{example}

The following theorem characterizes the facets of the triangulation of $\mathcal{C}_T$ for $T$ a tree under the specified good term order.

\begin{theorem}
    \label{thm: tree facet characterization}
    Let $S$ be a subset of generators of $R_T$ where $T$ is a tree on node set $[n+1]$. 
    Let $\overline{G}_{S,1},\ldots, \overline{G}_{S,M}$ be the $Z_S$-components of $G_S$. The set $S$ is a facet of the triangulation $\mathcal{T}$ of $\mathcal{C}_T$ corresponding to a lexicographic order induced by the order $<$ if and only if the following hold:
    \begin{enumerate}
        \item $G_S$ is connected,
        \item $G_S$ contains only single and double edges, where all double edges are of the form
        \begin{center}
    \begin{tikzpicture}[thick,scale=0.5]
        \tikzset{decoration={snake,amplitude=.4mm,segment length=2mm, post length=0mm,pre length=0mm}}
	
        \node[circle, draw = black!00, fill=black!00, inner sep=2pt, minimum width=2pt] (1) at (0,0) {};
        \node[circle, draw = black!00, fill=black!00, inner sep=2pt, minimum width=2pt] (2) at (2,0) {};

        \node[circle, draw = black!00, fill=black!00, inner sep=2pt, minimum width=2pt] (3) at (10,0) {};
        \node[circle, draw = black!00, fill=black!00, inner sep=2pt, minimum width=2pt] (4) at (12,0) {};

	\draw[->]   (1) edge[bend left] (2) ;
        \draw[-]    (1) -- (2) ;

        \draw[<-]   (3) edge[bend left] (4) ;
        \draw[-]    (3) -- (4) ;

        \node at (6,0) {or} ;
    \end{tikzpicture}
    \end{center}
    \item For all $j\in[M]$
    \begin{enumerate}
        \item $\overline{G}_{S,j}$ contains exactly $|V(\overline{G}_{S,j})\cap\mathfrak{Z}_S| - 1$ single edges,
        \item for all $i\in V(\overline{G}_{S,j})\cap(\mathfrak{Z}_S\setminus\{r_{S,j}\})$ the threshold path for $i$ is blocking,
        \item for all $i\in V(\overline{G}_{S,j})\cap(\mathfrak{Z}_S\setminus\{r_{S,j}\})$ there are no partially directed branchings to $i$ from a node that covers $i$, and 
        \item if $r_{S,j}\in\mathfrak{Z}_S$, then the edges incident to it are either a single $\middlewave{0.5cm}$, a single undirected edge, or an undirected edge together with an edge directed away from $r_{S,j}$. 
    \end{enumerate}
    \end{enumerate}
\end{theorem}

\begin{proof}
Assume first that $S$ is a facet.  
By Proposition~\ref{prop: connected} we know that (1) is satisfied.  
By Lemma~\ref{lem: facets means connected for trees} we know that (2) and (3)(a) are also satisfied. 

To see that (3)(b) holds, consider a leaf-to-root path from $i\in V(\overline{G}_{S,j})\cap(\mathfrak{Z}_S\setminus\{r_{S,j}\})$ and the portion of this path up to and including its first single edge. 
We note that this is a subpath of the threshold path for $i$.  
Since $S$ is a facet, all edges before the single edge are double, and by (2) they are of the specified form above. 
Since the path is leaf-to-root, then reading its vertices from the single edge out toward $i$ is an increasing sequence of nodes under the ordering $<_r$.  
Hence, if any of these double edges contains an arrow pointing toward $i$, then $G_S$ would contain a simple zig-zag obstruction of type 1, which is a contradiction to $S$ being a facet. 

Consider now the unique pair of vertices in $V(\overline{G}_{S,j})\cap\mathfrak{Z}_S$ separated by the single edge on this path, one of which is $i$ and the other of which we denote by $t$. 
Denote this path by $\pi$, and consider its associated $\alpha$-value, and the two paths $\pi_1$ and $\pi_2$ between $\alpha$ and $i$ and $\alpha$ and $t$. 
Here, we let $\pi_1$ denote the path between $\alpha$ and the least of the two vertices $i$ and $t$ under $<_r$, and $\pi_2$ denote the path between $\alpha$ and the largest of the two.
Note that the single edge is always on the path $\pi_1$ or $\pi_2$ that contains $i$.
Since the given single edge is the unique separator of $i$ and $t$, we know that all other edges on this path are double and of the form specified in (2). 
If $i <_r t$ and the single edge is on $\pi_1$ 
(the path from the least of the two nodes $i$ and $t$ to $\alpha$) then $\pi_2$ must consist of only double edge directed toward $\alpha$. 
Otherwise, $G_S$ would contain a simple zig-zag obstruction of type 1.  
Since all edges on $\pi_1$ are also double edges except for the single edge, $G_S$ would contain a simple zig-zag obstruction of type 2 if the single edge were undirected. 
Since it would contain a simple zig-zag obstruction of type 1 if the edge were directed toward $i$, the only valid options are a $\middlewave{0.5cm}$ or an edge directed toward $\alpha$.  
Hence, the threshold path for $i$ is of type 1 and blocking. 

On the other hand, if $t <_r i$, then the unique single edge on the path between $i$ and $t$ would lie on $\pi_2$ (i.e., the path from the largest of the two nodes $i$ and $t$ to $\alpha$).  
Hence, all edges on $\pi_1$ are double and directed toward $\alpha$ to avoid simple zig-zag obstructions of type 1.
We note that the single edge cannot be a directed edge toward $i$ as this would lead to a simple zig-zag obstruction of type 1 as well. 
If the edge is undirected, then $G_S$ contains no simple zig-zag obstructions of type $1$ only if all directed edges on $\pi_2$ point toward $\alpha$. 
If the single edge is $\middlewave{0.5cm}$, all directed edges between the single edge and $i$ must point toward $\alpha$.  
Finally, if the single edge is directed toward $\alpha$ there must be at least one directed edge between $\alpha$ and the single edge pointing toward $i$. 
Otherwise, $G_S$ would contain a simple zig-zag obstruction of type 2. 
Hence, the threshold path must also be blocking in this case. 
Therefore, (3)(b) holds for $S$ a facet. 

To see that (3)(c) holds, note that if there was such a partially directed branching then $G_S$ would necessarily contain a simple zig-zag obstruction of type 2, a contradiction to $S$ being a facet. 

To see that (3)(d) holds, note that any other choice of edge configuration at $r_{S,j}$ would imply that $G_S$ contains a fundamental obstruction. 
Hence, the listed conditions are all satisfied if $S$ is a facet of the triangulation. 

Conversely, suppose that the listed conditions are satisfied by $S$. 
Let $|Z_S| = n+1-k$.  
We will show now that (3)(a) implies that $G_S$ contains exactly $k$ double edges.  
Since $G_S$ is a connected graph on vertex set $[n+1]$, it contains at least $n$ edges.   
If exactly $k$ of these edges are double then $|S| = 2n+1$, which is the correct dimension for $S$ to be a facet. 
To see this claim, assume (3)(a) holds; i.e., assume that each $\overline{G}_{S,j}$ contains exactly $m_j-1$ single edges, where we let $m_j = |V(\overline{G}_{S,j})\cap\mathfrak{Z}_S|$. 
%
We induct on the number $M\geq1$ of $Z_S$-components in $G_S$.  
The result is seen to hold in the case that $M = 1$, as in this case $m_j - 1 = (n+1 - k) - 1$. 
Suppose now that the result holds for $G_S$ with at most $M-1 \geq 0$ $Z_S$-components, and consider $G_S$ with $M$ $Z_S$-components. 
Since $M > 1$, there is at least one $Z_S$-component containing a sink node of $\overrightarrow{T}$ that does not contain the root node $r$.  
Suppose this component is $\overline{G}_{S,j}$ and that its support graph has $n_j$ edges. 
Note that $r_{S,j}$ is necessarily a $\circ$ node, as $r_{S,j} \neq r$. 

We then have that the subgraph of $T$ given by deleting all edges of the support graph of $\overline{G}_{S,j}$ is a tree containing $n-n_j$ edges.  
Since $\overline{G}_{S,j}$ does not contain $r$ and does contain a sink node of $\overrightarrow{T}$, deleting all vertices in $V(\overline{G}_{S,j})\setminus \{r_{S,j}\}$ and all edges incident to these vertices from $G_S$, results in a graph $\widetilde{G}_{\tilde{S}}$ that contains $n+1-k-(m_j-1) = (n+1 - k) -m_j+1$ $\circ$ nodes. 
By the inductive hypothesis, $\widetilde{G}_{\widetilde{S}}$ contains $(n+1-k) - m_j$ single edges. 
By assumption, $\overline{G}_{S,j}$ contains $m_j-1$ single edges.  
Hence, $G_S$ contains $n-k$ single edges, or equivalently, $k$ double edges.  
%
%

Since $|S| = 2n+1$, which is the correct size for $S$ to be a facet, it only remains to see that the set $S$ contains no fundamental obstructions or simple zig-zag obstructions.  
The fact that $G_S$ contains no fundamental obstructions follows from (2) together with (3)(d) and (3)(b) (when we consider the definition of blocking).  
The fact that $G_S$ contains no simple zig-zag obstructions follows from (3)(b) and (3)(c), which completes the proof. 
\end{proof}


\section{Open problems}

We conclude with a few problems of interest left open by the article. As we worked out in the case of cycles and trees, a combinatorial analysis of the Gr\"{o}bner basis presented in \Cref{sec: grobner basis} reveals an explicit facet description for the corresponding triangulation. Moreover, one of the features of having a regular unimodular triangulation is that the computation of the volume can be reduced to counting the facets. It would be interesting to push this understanding further. 

\begin{problem}
    Obtain a facet description for a regular unimodular triangulation of the cosmological polytope of more general families of graphs. Can the volume of the cosmological polytope of an arbitrary graph be expressed in terms of elementary graph invariants?
\end{problem}

From a combinatorial viewpoint one is typically interested in finer invariants than the volume of a lattice polytope. 
One popular instance is the \emph{$h^*$-polynomial}; that is, the  univariate polynomial with integer coefficients which arises as the numerator of the \emph{Ehrhart series} of a lattice polytope (see for instance \cite[Chapter 3]{ccd}). 
As an example, we observed experimentally that the $h^*$-polynomial of the cosmological polytope of a tree on $n+1$ vertices equals $h^*(t)=(1+3t)^n$.

\begin{problem}
	Find formulas for the $h^*$-polynomial of a cosmological polytope $\mathcal{C}_G$ in terms of graph invariants of $G$.
\end{problem}

Finally, since, as we commented in the introduction, the computation of the canonical form of a polytope can be reduced to computing the canonical forms of the facets of any triangulation, we propose the following problem:

\begin{problem}
	Describe triangulations of cosmological polytopes with the minimum number of facets.
\end{problem}

Choosing a lexicographic term order for which $z$-variables are larger than the other variables, the corresponding initial ideal will contain all squares of  $z$-variables as generators, since each of them is the  leading term of some fundamental binomial. While the facets of the triangulation obtained in this way do not all have minimum volume, their number can be significantly smaller than in the unimodular case. For example, experiments suggest that this idea gives a triangulation of the cosmological polytope of a tree with $n+1$ vertices which consists of $2^{n-1}$ facets, compared to the $4^n$ many of a unimodular triangulation.

\subsection*{Acknowledgements}
The authors would like to thank Paolo Benincasa, Lukas K\"uhne and Leonid Monin for helpful discussions. 
We would also like to thank the \emph{2022 Combinatorial Coworkspace: A Session in Algebraic and Geometric Combinatorics} at Haus Bergkranz in Kleinwalsertal, Austria where this work began.
Liam Solus was supported by the Wallenberg
Autonomous Systems and Software Program (WASP) funded by the Knut and Alice Wallenberg Foundation, the Digital Futures Lab at KTH, the G\"oran Gustafsson Stiftelse Prize for Young Researchers, and Starting Grant No. 2019-05195
from The Swedish Research Council (Vetenskapsr\aa{}det).

\bibliographystyle{plain}
\bibliography{cosmo_bib}

\end{document}